\documentclass{amsart}

\usepackage[a4paper,hmargin=2.95cm,vmargin=4cm]{geometry}
\usepackage{amsfonts,amssymb,amscd,amstext}
\usepackage{graphicx}
\usepackage[dvips]{epsfig}
\usepackage[colorlinks=true,linkcolor=blue,citecolor=red]{hyperref}

\renewcommand{\leq}{\leqslant}
\renewcommand{\geq}{\geqslant}
\newcommand{\ptl}{\partial}
\newcommand{\rr}{{\mathbb{R}}}
\newcommand{\la}{\lambda}
\newcommand{\h}{\mathcal{H}}

\newcommand{\sub}{\subset}
\newcommand{\subeq}{\subseteq}
\newcommand{\escpr}[1]{\big<#1\big>}
\newcommand{\Sg}{\Sigma} \newcommand{\sg}{\sigma}
\newcommand{\Om}{\Omega}

\newcommand{\eps}{\varepsilon}
\newcommand{\var}{\varphi}
\newcommand{\ga}{\gamma}
\newcommand{\Ga}{\Gamma}
\newcommand{\mnh}{|N_{h}|}
\newcommand{\nuh}{\nu_{h}}

\newcommand{\e}{\mathbb{M}(\kappa)}
\newcommand{\stres}{\mathbb{S}^3}
\newcommand{\sph}{\mathbb{S}}
\newcommand{\mm}{\mathbb{M}}

\newcommand{\cmula}{\mathcal{C}_{\lambda}(\Ga)}
\newcommand{\sla}{\mathcal{S}_\lambda}

\DeclareMathOperator{\divv}{div}

\newtheorem{theorem}{Theorem}[section]
\newtheorem{proposition}[theorem]{Proposition}
\newtheorem{lemma}[theorem]{Lemma}
\newtheorem{corollary}[theorem]{Corollary}

\theoremstyle{definition}

\newtheorem{remark}[theorem]{Remark}
\newtheorem{remarks}[theorem]{Remarks}
\newtheorem{example}[theorem]{Example}

\theoremstyle{remark}

\numberwithin{equation}{section}

\begin{document}

\title[Instability criterion in sub-Riemannian $3$-space forms]
{An instability criterion for volume-preserving area-stationary surfaces with singular curves in sub-Riemannian $3$-space forms}

\author[A. Hurtado]{A. Hurtado}
\address{Departamento de Geometr\'{\i}a y Topolog\'{\i}a and Excellence Research Unit ``Modeling Nature'' (MNat), Universidad de Granada, E-18071,
Spain.}
 \email{ahurtado@ugr.es}

\author[C. Rosales]{C. Rosales}
\address{Departamento de Geometr\'{\i}a y Topolog\'{\i}a and Excellence Research Unit ``Modeling Nature'' (MNat) Universidad de Granada, E-18071,
Spain.} 
\email{crosales@ugr.es}

\date{\today}

\thanks{The authors were supported by MINECO grant No.~
MTM2017-84851-C2-1-P and Junta de Andaluc\'ia grant No.~FQM-325} 



\begin{abstract}
We study stable surfaces, i.e., second order minima of the area for variations of fixed volume, in sub-Riemannian space forms of dimension $3$. We prove a stability inequality and provide sufficient conditions ensuring instability of volume-preserving area-stationary $C^2$ surfaces with a non-empty singular set of curves. Combined with previous results, this allows to describe any complete, orientable, embedded and stable $C^2$ surface $\Sg$ in the Heisenberg group $\mathbb{H}^1$ and the sub-Riemannian sphere $\stres$ of constant curvature $1$. In $\mathbb{H}^1$ we conclude that $\Sg$ is a Euclidean plane, a Pansu sphere or congruent to the hyperbolic paraboloid $t=xy$. In $\stres$ we deduce that $\Sg$ is one of the Pansu spherical surfaces discovered in \cite{hr1}. As a consequence, such spheres are the unique $C^2$ solutions to the sub-Riemannian isoperimetric problem in $\stres$. 
\end{abstract}

\maketitle

\thispagestyle{empty}

\section{Introduction}
\label{sec:intro} 
\setcounter{equation}{0}

The main motivation of this work is to study the \emph{isoperimetric problem}, where we seek sets minimizing the perimeter under a volume constraint, in sub-Riemannian spaces. This is a global variational question that has received several contributions in the last years, especially in the Heisenberg group $\mathbb{H}^n$, which is the simplest model of a non-trivial sub-Riemannian manifold. It was conjectured by Pansu~\cite{pansu1} that the isoperimetric regions in $\mathbb{H}^1$ are bounded, up to congruence, by certain spherical $C^2$ surfaces with rotational symmetry and constant mean curvature in sub-Riemannian sense. This conjecture is supported by many partial results where further hypotheses involving regularity or symmetry of the solutions are assumed, see \cite[Ch.~8]{survey} and the Introduction of \cite{ritore-calibrations} for a precise description. Other related works are due to Monti~\cite{monti-rearrangements}, who has analyzed symmetrization in $\mathbb{H}^n$, Cheng, Chiu, Hwang and Yang~\cite{cchy}, who have characterized the Pansu spheres in $\mathbb{H}^n$ by using a notion of umbilicity, and Montefalcone~\cite{montefalcone-spheres}, who has derived some properties of the Pansu spheres as second order minima of the area under a volume constraint. 

Besides the Heisenberg group $\mathbb{H}^1$, which is the model by excellence of a simply connected flat sub-Riemannian $3$-manifold, it is also natural to investigate the isoperimetric problem in other \emph{$3$-dimensional space forms}, i.e., complete Sasakian sub-Riemannian $3$-manifolds of constant Webster curvature, see Section~\ref{subsec:ssR3m}. In the simply connected case, a result of Tanno~\cite{tanno} implies that, up to an isometry and a homothetic deformation of the sub-Riemannian metric, the unique $3$-space form of Webster curvature $\kappa$ is the model manifold $\e$ defined as the Heisenberg group $\mathbb{H}^1$ when $\kappa=0$, the group of unit quaternions $\sph^3\subset\rr^4$ when $\kappa=1$, and the universal cover of the special linear group $\text{SL}(2,\rr)$ when $\kappa=-1$. The spaces $\e$ are also the most symmetric examples of $3$-space forms; indeed, any simply connected and homogeneous contact sub-Riemannian $3$-manifold with isometry group of dimension $4$ is isometric to a model $\e$, see \cite{falbel1}.

Isoperimetric inequalities were discovered by Pansu~\cite{pansu2} in $\mm(0)$, and  by Chanillo and Yang~\cite{chanillo-yang} in $\mm(1)$ by means of a sub-Riemannian Santal\'o formula, see also the recent work of Prandi, Rizzi and Seri~\cite{santalo-sr}. The existence of isoperimetric regions in $\e$ is consequence of more general results by Leonardi and Rigot~\cite{lr} for Carnot groups, and by Galli and Ritor\'e~\cite{galli-ritore} for homogeneous manifolds. The regularity of the solutions is a difficult open problem, even in the Heisenberg group $\mm(0)$. As a matter of fact, though the conjectured solutions are bounded by $C^2$ surfaces, there exist area-minimizing surfaces in $\mm(0)$ with much less regularity, see \cite{pauls-regularity}, \cite{chy}, \cite{mscv} and \cite{r2}. 

In \cite{hr1,hr2} we constructed Pansu spherical surfaces in arbitrary $3$-space forms as union of Carnot-Carath\'eodory geodesics (CC-geodesics) connecting two given points. Motivated by some geometric and variational properties of these spheres we were naturally led to extend Pansu's conjecture to the model spaces $\e$. In this direction, it was shown in \cite{rr2} and \cite{hr2} that, in $\mm(\kappa)$ with $\kappa\leq 0$, any $C^2$ isoperimetric region is bounded by a Pansu sphere. The proof was based on a careful study of $C^2$ \emph{volume-preserving area-stationary surfaces}, which are the first order candidates to solve the isoperimetric problem. The main tools employed in this study are contained in Theorem~\ref{th:structure}, which gathers the ruling property of constant mean curvature surfaces, the structure of the \emph{singular set} consisting of the points where the surface is tangent to the horizontal plane, and the orthogonality property between the rulings and the singular curves. These allow to characterize any complete, oriented, volume-preserving area-stationary $C^2$ surface $\Sg$ with non-empty singular set $\Sg_0$ in any $3$-space form $M$. More precisely, it was proved in \cite{rr2}, \cite{hr1} and \cite{hr2} that $\Sg$ is either a Pansu sphere, an immersed plane with an isolated singular point, or a surface $\cmula$ obtained by leaving orthogonally from a complete CC-geodesic $\Ga$ in $M$ by CC-geodesics of curvature $\la$, see Section~\ref{subsec:cmula} for a detailed description. A culminating consequence is a sub-Riemannian version of Alexandrov's uniqueness theorem stating that a compact and connected volume-preserving area-stationary $C^2$ surface in $\mm(\kappa)$ with $\kappa\leq 0$ is a Pansu sphere. We must remark that in $\mm(1)$ the same theorem holds when we further assume that the surface is within an open hemisphere~\cite{hr2}. 

Indeed, in the sub-Riemannian $3$-sphere $\mm(1)$ the situation is very different and the classification of the $C^2$ isoperimetric surfaces does not follow only from the analysis of the critical ones. The difficulty here is that the family $\mathcal{F}$ of compact and connected volume-preserving area-stationary $C^2$ surfaces in $\mm(1)$ is considerably larger  than in the other model spaces. In order to place the results of this paper in a suitable context we need to recall some facts about the surfaces in $\mathcal{F}$. From the work of Cheng, Hwang, Malchiodi and Yang~\cite{chmy} any $\Sg\in\mathcal{F}$ is topologically a sphere or a torus. When $\Sg_0\neq\emptyset$ the aforementioned works imply that $\Sg$ is a Pansu sphere or a torus $\cmula$. In the case $\Sg_0=\emptyset$ the surface is a torus and there are only partial classification results \cite{hr1}. For instance, if the Hopf vector field in $\stres$ is always tangent to $\Sg$, or the mean curvature $H$ of $\Sg$ satisfies $H/\sqrt{1+H^2}\in\mathbb{R-\mathbb{Q}}$, then $\Sg$ is congruent to a vertical Clifford torus $\mathcal{T}_\rho:=\sph^1(\rho)\times\sph^1(\sqrt{1-\rho^2})$ with $\rho\in (0,1)$. After the characterization of constant mean curvature tori having rotational symmetry~\cite{hr1} or containing a vertical circle~\cite{hr3} the authors found embedded examples that are not congruent to $\mathcal{T}_\rho$. At summarizing, the family $\mathcal{F}$ contains Pansu spheres, surfaces $\cmula$, and an undetermined (possibly large) subfamily of constant mean curvature tori with empty singular set.

With the aim of discarding some surfaces in $\mathcal{F}$ as boundaries of isoperimetric regions in $\mm(1)$, we are led to consider the \emph{stability condition}, which means that the surface is a second order minimum of the area under deformations with fixed volume. There are many previous results concerning stable surfaces in $\mathcal{F}$. On the one hand, the second author established in~\cite{rosales} that any $\Sg\in\mathcal{F}$ with $\Sg_0=\emptyset$ is unstable. On the other hand, the authors analyzed in \cite{hr2} the stability properties of the Pansu spheres in any $3$-dimensional space form by proving, in particular, that they are all second order minima of the area for a large class of volume-preserving variations. As to the surfaces $\cmula$, they satisfy a strong stability condition under (possibly non-volume preserving) deformations \emph{supported on the regular set}~\cite{hr3}. Hence, to complete the study of stable surfaces in $\mathcal{F}$, it remains to check if $\cmula$ is stable for volume-preserving variations \emph{that possibly move the singular curves}. Indeed, the main contribution of the present work is to find a variation of this type to produce the instability of $\cmula$. This is not an easy task because $\cmula$ need not be compact, see examples in \cite{rr2} and \cite{hr1}, and the presence of the singular curves entails technical issues to compute the second derivative of the area. Since we are interested in surfaces bounding isoperimetric regions we will restrict ourselves to the case where $\cmula$ is also embedded. In this situation, our result in $\mm(1)$ comes from a more general instability criterion for arbitrary $3$-dimensional space forms. More precisely, in Theorem~\ref{th:main} we show the following: 
\begin{quotation}
\emph{In a $3$-dimensional space form $M$ of Webster curvature $\kappa$, an embedded $C^2$ surface $\cmula$ such that $\la^2+\kappa\geq 1$ and the length $\ell$ of $\Ga$ satisfies $\ell>\sqrt{2}\,\pi$ is unstable.}
\end{quotation}
In the model space $\mm(1)$ the length estimate $\ell>\sqrt{2}\,\pi$ holds for any CC-geodesic circle $\Ga$ by using the explicit expression of the CC-geodesics, see \cite{hr1} and also \cite[Prop.~2.5]{hr2}. Thus, by combining our instability criterion with the previous results, we deduce in Corollary~\ref{cor:stablesphere} that:
\begin{quotation}
\emph{The only complete, connected, oriented, embedded and stable $C^2$ surfaces in the sub-Riemannian $3$-sphere are the Pansu spheres.}
\end{quotation}
In the Heisenberg group $\mm(0)$ the instability criterion together with the existence of a one-parameter group of non-isotropic dilations implies the instability of all the embedded surfaces $\cmula$ when $\la\neq 0$. As to the surfaces $\mathcal{C}_0(\Ga)$, there are two possibilities. When $\Ga$ is a helix then $\mathcal{C}_0(\Ga)$ is a left-handed minimal helicoid, and we can adapt the proof in \cite[Thm.~5.4]{hrr} to conclude that $\mathcal{C}_0(\Ga)$ is unstable. If $\Ga$ is a horizontal line, then $\mathcal{C}_0(\Ga)$ is congruent to the $t$-graph $t=xy$, which is area-mininizing by a calibration argument, see \cite{rr2} and \cite{bscv}. Moreover, a complete stable $C^2$ surface with empty singular set in $\mm(0)$ must be a vertical plane~\cite{rosales}. As a consequence of all this, we obtain in Corollary~\ref{cor:stableh1} the following classification result: 
\begin{quotation}
\emph{The only complete, connected, oriented, embedded and stable $C^2$ surfaces in the first Heisenberg group are Euclidean planes, Pansu spheres, or surfaces $\mathcal{C}_0(\Ga)$ where $\Ga$ is a horizontal line.}
\end{quotation}
We must remark that the characterization in $\mm(0)$ of $C^2$ second order minima of the area \emph{without a volume constraint} was achieved in \cite{hrr}, see also \cite{dgnp-stable}. In this direction, Galli and Ritor\'e~\cite{galli-ritore3} have established the uniqueness of the vertical planes as complete stable area-stationary $C^1$ surfaces with empty singular set in $\mm(0)$.

The main tool for proving Theorem~\ref{th:main} is a \emph{stability inequality} $\mathcal{Q}(u)\geq 0$, that we derive in Theorem~\ref{th:staineq} for constant mean curvature surfaces with singular curves inside arbitrary Sasakian sub-Riemannian $3$-manifolds. The expression of $\mathcal{Q}(u)$ in equation \eqref{eq:indexform} defines a quadratic form, which involves analytic and geometric terms not only over the surface but also along the singular curves. In this way, we provide an extension of the inequality that Ritor\'e and the authors employed in \cite{hrr} to infer the instability of the left-handed minimal helicoids as area-stationary surfaces in $\mm(0)$. A similar inequality was used by Galli~\cite{galli,galli2} in his analysis of stable area-stationary surfaces in the roto-translation group and in the space of rigid motions of the Minkowski plane. 

Though the proof of Theorem~\ref{th:staineq} is inspired in the previous one in $\mm(0)$, some technical difficulties arise due to the volume-preserving condition in the stability notion, and the fact that the mean curvature need not vanish. Given a stable surface $\Sg$, we construct in a first step a volume-preserving variation of $\Sg$ with a prescribed velocity vector field. This is done with the help of Lemma~\ref{lem:bdc}, which is based on a result of Barbosa and do Carmo~\cite{bdc} for normal deformations of a Euclidean surface. The variation $\varphi$ in this lemma moves a neighborhood of the singular curves by vertical Riemannian geodesics, and the complementary set by a more complicated deformation, possibly with non-vanishing acceleration vector field. This is a remarkable difference with respect to previous works, where all the considered variations were based on Riemannian geodesics. Next, in order to use the stability of $\Sg$, we must compute the second derivative of the area $A''(0)$ for the variation $\var$. The calculus of $A''(0)$ off the singular curves is accomplished in Proposition~\ref{prop:2ndvar1} by means of a more general second variation formula in \cite{hr2}. For vertical deformations near the singular curves with vertical component constant along the rulings, the calculus of $A''(0)$ is contained in Proposition~\ref{prop:2ndvar2}. It is worth mentioning that the obtention of both $A''(0)$ and the inequality $\mathcal{Q}(u)\geq 0$ requires a careful study of $\Sg$ near the singular curves which is developed in Section~\ref{subsec:general}.

From the stability inequality, the proof of Theorem~\ref{th:main} relies on the delicate task of finding a mean zero function $u$ such that $\mathcal{Q}(u)<0$. From the geometric point of view, our test function is a suitable modification of the vertical component $\escpr{N,T}$ of the unit normal $N$ over the surface. The choice of such a function is motivated by the fact, proved in \cite{hr3}, that $\escpr{N,T}$ is a mean zero eigenfunction for the Jacobi operator on $\cmula$ that attains its extreme values along the singular curves. The embeddedness of $\cmula$ allows us to define $u$ in the coordinates $(\eps,s)$ which parameterize the different pieces of $\cmula$. We remark that $u\neq 0$ along two singular curves, so that the associated volume-preserving variation moves these curves. In the case where $\Ga$ is a circle, the hypothesis $\ell>\sqrt{2}\,\pi$ is combined with the classical Wirtinger's inequality to guarantee that $\mathcal{Q}(u)<0$. Though the optimality of this hypothesis is not clear for us, in Example~\ref{ex:cornucopia} we show that some kind of length estimate is necessary to deduce the instability of $\cmula$.

Coming back to our original motivation, in Section~\ref{sec:isoperimetric} we discuss the isoperimetric problem in $\mm(1)$. In Corollary~\ref{cor:isop} we prove that:
\begin{quotation}
\emph{Any $C^2$ isoperimetric regions in the $3$-sphere $\mm(1)$ is bounded by a Pansu sphere.}
\end{quotation}
This statement is a direct consequence of the classification of stable surfaces and a standard argument which implies the connectivity of isoperimetric boundaries. Hence, under $C^2$ regularity, the extended Pansu's conjecture is true in all the model spaces $\mm(\kappa)$. In spite of some advances about the critical points of the sub-Riemannian area with regularity less than $C^2$, see the references at the end of Section~\ref{subsec:stationary}, it is unknown if the conjecture holds for $C^1$ isoperimetric regions. 

In non-simply connected $3$-space forms the extended conjecture fails, as the authors discovered in \cite{hr2} a flat cylinder where the Pansu spheres do not always minimize the perimeter for fixed volume. This led us to conjecture that the isoperimetric property of these spheres must hold only for a range of volumes. To finish this work, we analyze the isoperimetric problem in the sub-Riemannian model $\mathbb{RP}^3$ of the $3$-dimensional projective space. In this space, there are solutions of any volume by compactness. In Corollary~\ref{cor:projective} we obtain that any $C^2$ isoperimetric region in $\mathbb{RP}^3$ is bounded by a Pansu sphere or an embedded torus $\cmula$, where $\Ga$ is a CC-geodesic circle. Moreover, a direct comparison shows that some Pansu spheres do not minimize.

The paper is organized into six sections. In Section~\ref{sec:preliminaries} we introduce the notation and gather some preliminary results. In Section~\ref{sec:cmula} we study volume-preserving area-stationary $C^2$ surfaces with singular curves. Section~\ref{sec:stability} is devoted to the second variational formulas for deformations moving the singular curves and the proof of the stability inequality. In the fifth section we establish the instability criterion for the surfaces $\cmula$ and deduce the classification of complete, embedded and stable $C^2$ surfaces in $\e$ with $\kappa\geq 0$. We conclude in Section~\ref{sec:isoperimetric} with our uniqueness results for isoperimetric regions in $\mm(1)$ and $\mathbb{RP}^3$.

\section{Preliminaries}
\label{sec:preliminaries} 
\setcounter{equation}{0} 

In this section we introduce some background material that will be used throughout the paper. This has been organized into several subsections.

\subsection{Sasakian sub-Riemannian $3$-manifolds}
\label{subsec:ssR3m}
\noindent

A \emph{contact sub-Riemannian $3$-manifold} is a connected $3$-manifold $M$ with $\ptl M=\emptyset$ together with a Riemannian metric $g_h$ defined on an oriented contact distribution $\h$, that is called \emph{horizontal distribution}. A vector field $U$ is \emph{horizontal} if $U_p\in\h_p$ for any point $p$ in the domain of $U$.

The \emph{normalized form} in $M$ is the contact $1$-form $\eta$ on $M$ such that $\text{Ker}(\eta)=\h$ and the restriction of the $2$-form $d\eta$ to $\h$ equals the area form in $\h$. Clearly $M$ is an orientable manifold:  we will always consider the orientation associated to the $3$-form $\eta\wedge d\eta$. The \emph{Reeb vector field} in $M$ is the vector field $T$ transversal to $\h$ given by equalities $\eta(T)=1$ and $d\eta(T,U)=0$, for any $U$. 

In the oriented planar distribution $\h$ with the Riemannian metric $g_h$ there is an orientation-preserving $90$ degree rotation that we denote by $J$. This is a contact structure on $\h$ since $J^2=-\text{Id}$. We extend $J$ to the whole tangent bundle of $M$ by setting $J(T):=0$. 

The \emph{canonical extension} of $g_h$ is the Riemannian metric $g=\escpr{\cdot\,,\cdot}$ on $M$ such that 
\[
g(U,V)=g_h(U,V), \quad g(T,U)=0, \quad g(T,T)=1,
\]
for any two horizontal vector fields $U$ and $V$. The \emph{norm} of a vector field $U$ is $|U|:=\escpr{U,U}^{1/2}$. We say that $M$ is \emph{complete} if $(M,g)$ is a complete Riemannian manifold. 

An \emph{isometry} between contact sub-Riemannian $3$-manifolds $M$ and $M'$ is a $C^\infty$ diffeomorphism $\phi:M\to M'$ whose differential at any $p\in M$ is an orientation-preserving linear isometry from $\h_p$ to $\h'_{\phi(p)}$. Two sets $\Sg_1,\Sg_2\subseteq M$ are \emph{congruent} if there is an isometry of $M$ such that $\phi(\Sg_1)=\Sg_2$.

A contact sub-Riemannian $3$-manifold $M$ where any diffeomorphism of the one-parameter group of $T$ is an isometry is a \emph{Sasakian sub-Riemannian $3$-manifold}. This is equivalent to that $(M,g)$ is a K-contact Riemannian manifold \cite[Cor.~6.3, Cor.~6.5]{blair}. Hence, for any vector field $U$ we have $D_UT=J(U)$, see \cite[Lem.~6.2]{blair}, where $D$ is the Levi-Civit\`a connection in $(M,g)$. Moreover, equality
\begin{equation}
\label{eq:dujv}
D_U\left(J(V)\right)=J(D_UV)+\escpr{V,T}\,U-\escpr{U,V}\,T
\end{equation}
holds for any pair $U,V$ of vector fields.

The \emph{Webster $($scalar$)$ curvature} of a contact sub-Riemannian $3$-manifold $M$ is the sectional curvature $K$ of the horizontal distribution $\h$ with respect to the Tanaka connection. For Sasakian manifolds this is related to the sectional curvature $K_h$ of $\h$ in $(M,g)$ by means of the equality $K=(1/4)\,(K_h+3)$, see \cite[Sect.~10.4]{blair}.

By a \emph{$3$-dimensional space form} we mean a complete Sasakian sub-Riemannian $3$-manifold $M$ of constant Webster curvature $\kappa$. If $M$ is simply connected and $\kappa\in\{-1,0,1\}$, then a result of Tanno~\cite{tanno} implies that $M$ is isometric to a model space $\e$. The space $\e$ is the first Heisenberg group $\mathbb{H}^1$ for $\kappa=0$, the group of unit quaternions $\stres\sub\rr^4$ for $\kappa=1$, and the universal cover of the special linear group $\text{SL}(2,\rr)$ for $\kappa=-1$. We refer the reader to \cite[Sect.~2.2]{rosales} and \cite[Sect.~2.2]{hr2} for precise descriptions of $\e$ and other non-simply connected space forms.

\subsection{Carnot-Carath\'eodory geodesics and Jacobi fields}
\label{subsec:geodesics}
\noindent

A \emph{horizontal curve} in a Sasakian sub-Riemannian $3$-manifold $M$ is a $C^1$ curve $\ga:I\to M$, defined on an interval $I\subeq\rr$, and with horizontal velocity vector $\dot{\ga}$. The \emph{length} of $\ga$ in a compact interval $[a,b]\subseteq I$ is $\int_a^b|\dot{\ga}(s)|\,ds$. Following the approach in \cite[Sect.~3]{rr2} and \cite[Sect.~3]{rosales}, we say that a $C^2$ horizontal curve $\ga$ parameterized by arc-length is a \emph{$CC$-geodesic} if it is a critical point of length under $C^2$ variations by horizontal curves. As in \cite[Prop.~3.1]{rr2} this is equivalent to the existence of a constant $\la\in\rr$, called the \emph{curvature} of $\ga$, such that the second order ODE
\begin{equation}
\label{eq:geoeq}
\dot{\ga}'+2\la\,J(\dot{\ga})=0
\end{equation}
is satisfied. Here the prime $'$ denotes the covariant derivative along $\ga$ in $(M,g)$. It follows that any CC-geodesic in $M$ is a $C^\infty$ curve. If $p\in M$ and $w\in\h_p$ with $|w|=1$, then the unique maximal solution $\ga$ to the geodesic equation \eqref{eq:geoeq} with $\ga(0)=p$ and $\dot{\ga}(0)=w$ is a CC-geodesic of curvature $\la$ since $\escpr{\dot{\ga},T}$ and $|\dot{\ga}|^2$ are constant functions along $\ga$. It is known that, if $M$ is complete, then any maximal CC-geodesic in $M$ is defined on $\rr$, see for instance ~\cite[Thm.~1.2]{falbel4}.

As in Riemannian geometry, the notion of \emph{CC-Jacobi field} appears when one considers the variational vector field associated to a one-parameter family of CC-geodesics of the same curvature, see \cite[Lem.~3.5]{rr2} and \cite[Lem.~3.3]{rosales}. In the next result, which follows from \cite[Lem.~3.3, Lem.~3.4]{rosales}, we gather some facts about CC-Jacobi fields that will be useful in this work.

\begin{lemma}
\label{lem:ccjacobi}
Let $M$ be a Sasakian sub-Riemannian $3$-manifold. Consider a $C^1$ curve $\Gamma:I\to M$ defined on some open interval $I\subseteq\rr$, and a unit horizontal $C^1$ vector field $U$ along $\Gamma$. For a fixed $\lambda\in\rr$, suppose that we have a well-defined map $F:I\times I'\to M$ given by $F(\eps,s):=\ga_{\eps}(s)$, where $I'$ is an open interval containing $0$, and $\ga_{\eps}(s)$ is the CC-geodesic of curvature $\la$ in $M$ with $\ga_{\eps}(0)=\Gamma(\eps)$ and $\dot{\ga}_{\eps}(0)=U(\eps)$. Then, the CC-Jacobi vector field $X_{\eps}(s):=(\ptl F/\ptl\eps)(\eps,s)$ and the function $v_\eps(s):=\escpr{X_\eps(s),T}$ satisfy these properties:
\begin{itemize}
\item[(i)] $X_\eps$ is $C^\infty$ along $\ga_\eps$ with $[\dot{\ga}_\eps,X_\eps]=0$,
\item[(ii)] the expression of $X_\eps$ with respect to the orthonormal basis 
$\{\dot{\ga}_\eps,J(\dot{\ga}_\eps),T\}$ is
\[
X_\eps=\big\{\la\,\big(\escpr{\dot{\Ga}(\eps),T}-v_\eps\big)+\escpr{\dot{\Ga}(\eps),U(\eps)}\big\}\,\dot{\ga}_\eps+(v'_\eps/2)\,J(\dot{\ga}_\eps)+v_\eps\,T,
\]
where the prime $'$ stands for the derivative with respect to $s$,
\item[(iii)] the function $v_\eps$ satisfies the differential equation $v'''_\eps+\tau\,v'_\eps=0$, where $\tau:=4\,(\la^2+K)$. In particular, if $K$ is constant and $\tau>0$, then we have:
\[
v_\eps(s)=\frac{1}{\sqrt{\tau}}\left(a_\eps\,\sin(\sqrt{\tau}\,s)-b_\eps\,\cos(\sqrt{\tau}\,s)\right)+c_\eps,
\]
where $a_\eps=v_\eps'(0)$, $b_\eps=(1/\sqrt{\tau})\,v_\eps''(0)$ and $c_\eps=v_\eps(0)+(1/\tau)\,v_\eps''(0)$.
\end{itemize}
\end{lemma}

\subsection{Horizontal geometry of surfaces}
\label{subsec:surfaces}
\noindent

Let $M$ be a Sasakian sub-Riemannian  $3$-manifold and $\Sg$ a $C^1$ surface immersed in $M$. Unless explicitly stated we always assume that $\ptl\Sg=\emptyset$. We say that $\Sg$ is \emph{complete} if it is complete with respect to the Riemannian metric induced by $g$.

The \emph{singular set} of $\Sg$ is the set $\Sg_0$ of the points $p\in\Sg$ where the tangent plane $T_p\Sg$ equals the horizontal plane $\h_{p}$. Since a contact distribution is completely nonintegrable, it follows by Frobenius theorem that $\Sg_0$ is closed and has empty interior in $\Sg$. Hence the \emph{regular set} $\Sg-\Sg_0$ is open and dense in $\Sg$. From the arguments in \cite[Lem.~1]{d2}, see also \cite[Thm.~1.2]{balogh} and \cite[App.~A]{hp2}, the Hausdorff dimension of $\Sg_{0}$ in $(M,g)$ is less than or equal to $1$ for any $C^2$ surface $\Sg$. In particular, the Riemannian area of $\Sg_{0}$ vanishes.

If $\Sg$ is orientable and we choose a unit vector field $N$ normal to $\Sg$ in $(M,g)$, then we have $\Sg_{0}=\{p\in\Sg\,;N_h(p)=0\}$, where $N_h$ denotes the horizontal projection of $N$. Thus, in the regular set $\Sg-\Sg_0$, we can define the \emph{horizontal Gauss map} $\nu_h$ and the \emph{characteristic vector field} $Z$ by
\begin{equation}
\label{eq:nuh}
\nu_h:=\frac{N_h}{|N_h|}, \qquad Z:=J(\nuh).
\end{equation}
As $Z$ is horizontal and orthogonal to $\nu_h$ then $Z$ is tangent to $\Sg$. Hence $Z_{p}$ generates $T_{p}\Sg\cap\h_{p}$ for any $p\in\Sg-\Sg_0$. We call $(\emph{oriented}\,)$ \emph{characteristic curves} of $\Sg$ to the integral curves of $Z$ in $\Sg-\Sg_0$. These curves are tangent to $\Sg$ and horizontal. If we define
\begin{equation}
\label{eq:ese}
S:=\escpr{N,T}\,\nu_h-|N_h|\,T,
\end{equation}
then $\{Z_{p},S_{p}\}$ is an orthonormal basis of $T_p\Sg$ whenever
$p\in\Sg-\Sg_0$. Hence, we deduce that
\begin{equation}
\label{eq:relations}
\nu_{h}^\top=\escpr{N,T}\,S,
\qquad T^\top=-|N_{h}|\,S,
\end{equation}
on $\Sg-\Sg_0$, where $U^\top$ stands for the projection of a vector field $U$ onto the tangent plane to $\Sg$.

Suppose now that $\Sg$ is an orientable $C^2$ surface immersed in $M$. For 
$p\in\Sg-\Sg_0$ and $U\in T_pM$, these equalities are easy to prove, see \cite[Lem.~3.5]{hrr} for the details
\begin{align}
\label{eq:vder}
U(\escpr{N,T})&=\escpr{D_{U}N,T}+\escpr{N,J(U)},\quad U\,(|N_h|)=\escpr{D_{U}N, \nu_{h}}+\escpr{N,T}\,\escpr{U,Z},
\\
\label{eq:dvnuh}
D_{U}\nu_h&=|N_h|^{-1}\, \big(\escpr{D_UN,Z}-\escpr{N,T}\,
\escpr{U,\nuh}\big)\,Z+\escpr{U,Z}\,T.
\end{align}
The \emph{shape operator} $B$ of $\Sg$ in $(M,g)$ is given by $B(U):=-D_{U}N$, for any vector $U$ tangent to $\Sg$. As in \cite{rr2} and \cite{rosales} we define the \emph{$($sub-Riemannian$)$ mean curvature} of $\Sg$ by equality
\begin{equation}
\label{eq:mc}
-2H(p):=(\divv_\Sg\nuh)(p),\quad p\in\Sg-\Sg_0,
\end{equation}
where $\divv_\Sg U$ is the divergence relative to $\Sg$ in $(M,g)$ of a $C^1$ vector field $U$ on $\Sg$. The next formulas involving the mean curvature and the shape operator will be frequently used in this work:
\begin{align}
\label{eq:mc2}
\escpr{B(Z),Z}&=2H\,\mnh,
\\
\label{eq:zder}
Z(\escpr{N,T})&=\mnh\,\big(\escpr{B(Z),S}-1\big),\quad Z(\mnh)=\escpr{N,T}\,\big(1-\escpr{B(Z),S}\big),
\\
\label{eq:sder}
S(\escpr{N,T})&=\mnh\,\escpr{B(S),S}, \quad S(\mnh)=-\escpr{N,T}\,\escpr{B(S),S},\quad 
\\
\label{eq:divzs}
\divv_\Sg Z&=\mnh^{-1}\,\escpr{N,T}\,\big(1+\escpr{B(Z),S}\big),\quad \divv_\Sg S=-2H\,\escpr{N,T}.
\end{align}
The equalities in \eqref{eq:zder} and \eqref{eq:sder} follow from \eqref{eq:vder} and \eqref{eq:relations}. Those in \eqref{eq:divzs} come from \cite[Lem.~5.5]{rosales}. On the other hand, equation \eqref{eq:dvnuh} implies that $D_Z\nuh=T-\mnh^{-1}\,\escpr{B(Z),Z}\,Z$ and that $D_S\nuh$ is proportional to $Z$. Hence \eqref{eq:mc2} is obtained from \eqref{eq:mc} when we compute the divergence $\divv_\Sg\nuh$ by using the orthonormal basis $\{Z,S\}$.

\subsection{Volume-preserving area-stationary surfaces}
\label{subsec:stationary}
\noindent

Let $M$ be a Sasakian sub-Riemannian $3$-manifold and $\varphi_0:\Sg\to M$ an oriented $C^2$ surface immersed in $M$. Following \cite{rr2} and \cite{rosales}, we define the (sub-Riemannian) \emph{area} of $\Sg$ by
\[
A(\Sg):=\int_{\Sg}|N_{h}|\,da,
\]
where $N$ is the Riemannian unit normal compatible with the orientations of $\Sg$ and $M$, and $da$ is the area element in $(M,g)$. This definition is also valid for an oriented $C^1$ surface $\Sg$. 

By a \emph{$($\!compactly supported$)$ variation} of $\Sg$ we mean a map $\varphi:I\times\Sg\to M$ (which we assume to be $C^2$ unless otherwise stated) defined for some open interval $I\subeq\rr$ with $0\in I$, and satisfying:
\begin{itemize}
\item[(i)] $\varphi(0,p)=\varphi_0(p)$ for any $p\in\Sg$,
\item[(ii)] the map $\varphi_r:\Sg\to M$ given by $\varphi_r(p):=\varphi(r,p)$ is an immersion for any $r\in I$, 
\item[(iii)] there is a compact set $C\subseteq\Sg$ such that $\varphi_{r}(p)=\varphi_0(p)$ for any $r\in I$ and $p\in\Sg-C$.
\end{itemize}
We denote by $\Sg_r$ the immersed surface induced by the map $\var_r:\Sg\to M$. It is clear that $\Sg_r-C=\Sg-C$. For any $p\in\Sg$ we consider the curve $\ga_p(r):=\var_r(p)$ with $r\in I$. The \emph{velocity and acceleration} associated to the variation are the vector fields $U$ and $W$ such that
\[
U_p:=\dot{\ga}_p(0),\quad W_p:=\dot{\ga}'_p(0), \quad p\in\Sg.
\]

The \emph{area functional} $A:I\to\rr$ is the function 
\begin{equation}
\label{eq:duis}
A(r):=A(\Sg_r)=\int_{\Sg}\mnh_r(p)\,|\text{Jac}\,\varphi_r|(p)\,da, \quad r\in I.
\end{equation}
Here $\mnh_r(p):=\mnh\big(\ga_p(r)\big)$, where $N$ stands for a $C^1$ vector field along the variation that coincides, for any $r\in I$, with the Riemannian unit normal $N_r$ along the immersion $\var_r:\Sg\to M$ which is compatible with the orientations of $\Sg$ and $M$. On the other hand, if $p\in\Sg$ and $\{e_1,e_2\}$ is an orthonormal basis in $T_p\Sg$, then $|\text{Jac}\,\varphi_r|(p):=\big(\text{det}\,G(r)\big)^{1/2}$, where $G(r)$ is the matrix with entries $\escpr{e_i(\varphi_r),e_j(\varphi_r)}$ with $i,j=1,2$. In \eqref{eq:duis} we understand that the integral is computed in the compact set $C$, so that $A(r)$ is finite and measures the area of $\var_r(C)$. Moreover, since $\Sg_0$ has vanishing Riemannian area we can replace $C$ with $C-\Sg_0$.

Now we define a \emph{volume functional} $V(r)$ associated to the variation $\var$. Since the surfaces $\Sg_r$ need not be compact nor embedded, we consider the \emph{signed volume enclosed between $\Sg$ and $\Sg_r$}, see \cite[Sect.~2]{bdce}. In precise terms, if we denote by $dv$ the volume element in $(M,g)$, then we have
\begin{equation}
\label{eq:volume}
V(r):=\int_{[0,r]\times C}\varphi^*(dv).
\end{equation}
The variation $\var$ is \emph{volume preserving} if $V(r)$ is a constant function. 

We say that the surface $\Sg$ is \emph{volume-preserving area-stationary} if $A'(0)=0$ for any volume-preserving variation. As a well-known consequence of the first variational formulas for area and volume such a surface has constant mean curvature, see for instance \cite[Sect.~4.1]{hr2} and the references therein. This means that the function $H$ in \eqref{eq:mc} is constant on the set $\Sg-\Sg_0$. When $H=0$ the surface $\Sg$ is called \emph{minimal}.

The next result gathers the main properties of CMC surfaces that we need in this work.

\begin{theorem}
\label{th:structure}
Let $\Sg$ be an oriented $C^2$ surface of constant mean curvature $H$ immersed in a Sasakian sub-Riemannian $3$-manifold $M$. Then, we have:
\begin{itemize}
\item[(i)] $($\cite{chmy,hp1,rosales}$)$ any characteristic curve of $\Sg$ is a CC-geodesic in $M$ of curvature $H$,
\item[(ii)] $($\cite[Thm.~B]{chmy}, \cite[Sect.~5]{galli}$)$ the singular set $\Sg_0$ consists of isolated points and $C^1$ curves with non-vanishing tangent vector (singular curves),
\item[(iii)] $($\cite[Prop.~3.5, Cor.~3.6]{chmy}$)$ if $p$ is contained in a $C^1$ curve $\Ga\subseteq\Sg_{0}$, then there is a neighborhood $\mathcal{D}$ of $p$ in $\Sg$ such that $\mathcal{D}-\Gamma$ is the union of two disjoint connected open sets $\mathcal{D}^+, \mathcal{D}^-\subsetneq\Sg-\Sg_0$. For any $q\in\Ga\cap\mathcal{D}$ there are exactly two CC-geodesics $\ga_{1}\sub \mathcal{D}^+$ and $\ga_{2}\sub \mathcal{D}^-$ of curvature $\la$ leaving from $q$ and meeting transversally $\Ga$ at $q$ with opposite initial velocities. If $N_p=T_p$ then $\la=H$ and the CC-geodesics $\ga_i$, $i=1,2$, are characteristic curves of $\Sg$. If $N_p=-T_p$ then $\la=-H$. 
\item[(iv)] $($\cite[Thm.~4.17, Prop.~4.20]{rr2}, \cite[Cor.~5.4]{galli}$)$ the surface $\Sg$ is volume-preserving area-stationary if and only if the characteristic curves meet orthogonally the singular curves when they exist. Moreover, in such a case, any singular curve in $\Sg$ is a $C^2$ curve.
\end{itemize}
\end{theorem}

The statement (i) above is known as the \emph{ruling property} of CMC surfaces. This is a direct consequence of \eqref{eq:geoeq} and equality
\begin{equation}
\label{eq:dzz}
D_ZZ=(2H)\,\nuh\quad\text{on } \Sg-\Sg_0.
\end{equation}
Given a point $p\in\Sg-\Sg_0$, it is clear that $\{Z_p,(\nuh)_p,T_p\}$ is an orthonormal basis of $T_pM$. Thus, the vector field $D_ZZ$ is proportional to $\nuh$ since $\escpr{D_ZZ,Z}=0$ and $\escpr{D_ZZ,T}=-\escpr{Z,J(Z)}=0$. From \eqref{eq:dvnuh} and \eqref{eq:mc2} we get $D_Z\nuh=T-(2H)\,Z$, which proves \eqref{eq:dzz}.

For oriented CMC surfaces of class $C^1$, the regularity of the characteristic curves, the ruling property of the regular set and the description of the singular set are much more involved than in the $C^2$ case, see \cite{chy2}, \cite{chmy2}, \cite{galli-ritore2} and \cite{galli3}. There are also generalizations of Theorem~\ref{th:structure} (iv) involving other sub-Riemannian settings and/or lower regularity hypotheses, see \cite{chy}, \cite{ch2}, and \cite{hp2}.

\section{Stationary surfaces with singular curves}
\label{sec:cmula}
\setcounter{equation}{0} 

In this section we study in more detail volume-preserving area-stationary surfaces having at least one singular curve. We first consider arbitrary Sasakian sub-Riemannian $3$-manifolds, where we derive some useful computations for Section~\ref{sec:stability}. Later we will obtain properties of these surfaces, specially in the embedded case, when the ambient manifold is a $3$-dimensional space form. Our analysis will be necessary to prove the instability result in Section~\ref{sec:main}.

\subsection{The general case}
\label{subsec:general}
\noindent

According to Theorem~\ref{th:structure} (iii), around any point in a singular curve, a volume-preserving area-stationary $C^2$ surface is union of CC-geodesics segments of the same curvature leaving orthogonally from the curve. This motivates the next construction where, for a given horizontal curve, we produce CMC neighborhoods foliated by orthogonal CC-geodesic rays of the same length.

Let $M$ be a Sasakian sub-Riemannian $3$-manifold, $\Ga:I\to M$ a $C^3$ horizontal curve parameterized by arc-length, and $\la\in\rr$. For any  $i\in\{1,2\}$ and $\eps\in I$, we take the CC-geodesic $\ga_{i,\eps}(s)$ in $M$ of curvature $\la$ with $\ga_{i,\eps}(0)=\Ga(\eps)$ and $\dot{\gamma}_{i,\eps}(0)=(-1)^{i-1}\,J(\dot{\Ga}(\eps))$. We suppose that there are numbers $s_i>0$ with $i=1,2$ such that the $C^2$ maps $F_i:I\times[0,s_i]\to M$ given by $F_i(\eps,s):=\ga_{i,\eps}(s)$ are well-defined immersions. We define the immersed surfaces
\begin{equation}
\label{eq:sigmaila}
\Sg_{i,\la}(\Ga):=F_i(I\times[0,s_i])=\{\ga_{i,\eps}(s)\,;\,\eps\in I,\,s\in[0,s_i]\},
\end{equation}
and the functions $v_{i,\eps}(s):=\escpr{X_{i,\eps}(s),T}$, where $X_{i,\eps}(s):=(\ptl F_i/\ptl\eps)(\eps,s)$. We will denote by primes $'$ the derivatives of functions depending on $s$ and the covariant derivatives of vector fields along $\Ga$. 

In the next lemma we compute and analyze the behaviour near $\Ga$ of some geometric quantities on $\Sg_{i,\la}(\Ga)$. In the model spaces $\mm(\kappa)$ with $\kappa\in\{0,1\}$ some of the statements below were proved in \cite[Prop.~6.3, Re.~6.5]{rr2} and \cite[Prop.~5.5, Re.~5.6]{hr1} by using the explicit expression of the CC-geodesics. The extension property in (v) was also proved in \cite[Prop.~3.5]{chmy}.

\begin{lemma}
\label{lem:coor}
In the previous situation, we have:
\begin{itemize}
\item[(i)] $v_{i,\eps}(s)$ is a $C^\infty$ function of $s$ with
\[
v_{i,\eps}(0)=0,\quad v_{i,\eps}'(0)=2\,(-1)^{i}, \quad v_{i,\eps}''(0)=2\,h(\eps),
\]
where $h(\eps):=\escpr{\dot{\Ga}'(\eps),J(\dot{\Ga}(\eps))}$,
\item[(ii)] a point $p=F_i(\eps,s)$ belongs to the singular set of $\Sg_{i,\la}(\Ga)$ if and only if $v_{i,\eps}(s)=0$. In particular $\Ga$ is a singular curve of $\Sg_{i,\la}(\Ga)$.
\end{itemize}
Furthermore, if $v_{i,\eps}(s)\neq 0$ for any $s\in(0,s_i)$, then:
\begin{itemize}
\item[(iii)] there is a Riemannian unit normal $N_i$ on $\Sg_{i,\la}(\Ga)$ such that $N_i=T$ along $\Ga$, and any CC-geodesic $\ga_{i,\eps}(s)$ with $s\in (0,s_i)$ is a characteristic curve of $\Sg_{i,\la}(\Ga)$. In particular $\Sg_{i,\la}(\Ga)$ has constant mean curvature $\la$ with respect to $N_i$,
\item[(iv)] in the coordinates $(\eps,s)\in I\times (0,s_i)$ these equalities hold
\begin{align*}
da_i&=\frac{\sqrt{4\,v_{i,\eps}(s)^2+v_{i,\eps}'(s)^2}}{2}\,d\eps\,ds,
\\
|(N_i)_h|(\eps,s)&=\frac{2\,(-1)^{i}\,v_{i,\eps}(s)}{\sqrt{4\,v_{i,\eps}(s)^2+v_{i,\eps}'(s)^2}}, \qquad \escpr{N_i,T}(\eps,s)=\frac{(-1)^i\,v'_{i,\eps}(s)}{\sqrt{4\,v_{i,\eps}(s)^2+v_{i,\eps}'(s)^2}},
\\
S_i(\eps,s)&=\frac{2\,(-1)^{i-1}}{\sqrt{4\,v_{i,\eps}(s)^2+v_{i,\eps}'(s)^2}}\,X_{i,\eps}(s)-\la\,|(N_i)_h|(\eps,s)\,Z_i(\eps,s),
\\
\escpr{B(Z_i),S_i}(\eps,s)&=\frac{2\,v_{i,\eps}(s)\,v_{i,\eps}''(s)+4\,v_{i,\eps}(s)^2-v'_{i,\eps}(s)^2}{4\,v_{i,\eps}(s)^2+v'_{i,\eps}(s)^2},
\end{align*}
where $da_i$ is the Riemannian area element on $\Sg_{i,\la}(\Ga)$ and $\{Z_i,S_i\}$ is the tangent orthonormal basis defined in \eqref{eq:nuh} and \eqref{eq:ese}.
\item[(v)] the vector field $S_i$ extends continuously to $\Ga$ in such a way that $S_1=\dot{\Ga}=-S_2$ along $\Ga$.
\item[(vi)] the functions $\escpr{B(Z_i),S_i}$ and $q_i:=|B(Z_i)+S_i|^2+4\,(K-1)\,|(N_i)_h|^2$, where $K$ is the Webster curvature of $M$, satisfy
\[
\lim_{\eps\to\eps_0,\,s\to 0}\escpr{B(Z_i),S_i}(\eps,s)=-1, \quad \lim_{\eps\to\eps_0\,s\to 0}\big(|(N_i)_h|^{-1}\,q_i\big)(\eps,s)=0,
\]
for any $\eps_0\in I$.
\end{itemize}
\end{lemma}

\begin{proof}
From Lemma~\ref{lem:ccjacobi} (i) we know that $X_{i,\eps}$ is a $C^\infty$ vector field along $\ga_{i,\eps}$. Hence $v_{i,\eps}(s)$ is a $C^\infty$ function of $s$. Note that $X_{i,\eps}(0)=\dot{\Ga}(\eps)$, and so $v_{i,\eps}(0)=0$ since $\Ga$ is horizontal. By using Lemma~\ref{lem:ccjacobi} (ii) we get the equality
\begin{equation}
\label{eq:Vieps}
X_{i,\eps}=-(\la\,v_{i,\eps})\,\dot{\ga}_{i,\eps}+(v_{i,\eps}'/2)\,J(\dot{\ga}_{i,\eps})+v_{i,\eps}\,T,
\end{equation}
from which $v'_{i,\eps}(0)=2\,(-1)^i$. By differentiating with respect to $s$, we deduce
\[
X_{i,\eps}'(0)=2\la\,J(\dot{\Ga}(\eps))+\frac{(-1)^i}{2}\,v_{i,\eps}''(0)\,\dot{\Ga}(\eps)+(-1)^iJ(\dot{\ga}_{i,\eps})'(0)+2\,(-1)^i\,T_{\Ga(\eps)}.
\]
Note that $J(\dot{\ga}_{i,\eps})'=2\la\,\dot{\ga}_{i,\eps}-T$ by \eqref{eq:dujv} and \eqref{eq:geoeq}. Hence
\[
X_{i,\eps}'(0)=(-1)^i\,\bigg(\frac{v_{i,\eps}''(0)}{2}\,\dot{\Ga}(\eps)+T_{\Ga(\eps)}\bigg).
\]
On the other hand, Lemma~\ref{lem:ccjacobi} (i) implies $[\dot{\ga}_{i,\eps},X_{i,\eps}]=0$, so that $X_{i,\eps}'=D_{\dot{\ga}_{i,\eps}}X_{i,\eps}=D_{X_{i,\eps}}\dot{\ga}_{i,\eps}$ along $\ga_{i,\eps}$. As a consequence
\[
X_{i,\eps}'(0)=(-1)^{i-1}\,D_{X_{i,\eps}}J(\dot{\Ga})=(-1)^{i-1}\big(J(\dot{\Ga}'(\eps))-T_{\Ga(\eps)}\big)=(-1)^i\big(h(\eps)\,\dot{\Ga}(\eps)+T_{\Ga(\eps)}\big),
\]
where we have employed \eqref{eq:dujv} and that $\dot{\Ga}'=h\,J(\dot{\Ga})$. The two previous equalities for $X_{i,\eps}'(0)$ yield $v''_{i,\eps}(0)=2\,h(\eps)$. This proves (i). Statement (ii) follows since the tangent plane to $\Sg_{i,\la}(\Ga)$ at $p=F_i(\eps,s)$ is generated by the vectors $(\ptl F_i/\ptl s)(\eps,s)=\dot{\ga}_{i,\eps}(s)$ and $X_{i,\eps}(s)$. 

To obtain (iii) observe that the map
\begin{equation}
\label{eq:normalcoor}
N_i(\eps,s):=(-1)^{i-1}\,\frac{2\,v_{i,\eps}(s)\,J(\dot{\ga}_{i,\eps}(s))-v_{i,\eps}'(s)\,T}{\sqrt{4\,v_{i,\eps}(s)^2+v_{i,\eps}'(s)^2}}
\end{equation}
defines a unit normal vector to $\Sg_{i,\la}(\Ga)$ at $F_i(\eps,s)$ with $N_i(\eps,0)=T_{\Ga(\eps)}$. As we assume that $v_{i,\eps}$ never vanishes in $(0,s_i)$, the equalities $v_{i,\eps}(0)=0$ and $v'_{i,\eps}(0)=2\,(-1)^i$ imply that $(-1)^i\,v_{i,\eps}>0$ in $(0,s_i)$. From~\eqref{eq:normalcoor} it is easy to check that the associated characteristic field satisfies $Z_i(\eps,s)=\dot{\ga}_{i,\eps}(s)$. Thus any CC-geodesic $\ga_{i,\eps}(s)$ with $s\in (0,s_i)$ is a characteristic curve of $\Sg_{i,\la}(\Ga)$. By equations \eqref{eq:dzz} and \eqref{eq:geoeq} this implies that $\Sg_{i,\la}(\Ga)$ has constant mean curvature $\la$. 

Let us prove (iv). From equation \eqref{eq:Vieps} we have
\[
da_i=\big(|X_{i,\eps}|^2-\escpr{X_{i,\eps},\dot{\ga}_{i,\eps}}^2\big)^{1/2}\,d\eps\,ds=\frac{\sqrt{4\,v_{i,\eps}(s)^2+v_{i,\eps}'(s)^2}}{2}\,d\eps\,ds.
\]
The announced formulas for $|(N_i)_h|$ and $\escpr{N_i,T}$ come immediately from \eqref{eq:normalcoor}. As a consequence, the associated tangent vector field $S_i$ in \eqref{eq:ese} is
\[
S_i(\eps,s)=\frac{(-1)^{i+1}\,v'_{i,\eps}(s)}{\sqrt{4\,v_{i,\eps}(s)^2+v_{i,\eps}'(s)^2}}\,J(\dot{\ga}_{i,\eps}(s))+\frac{2\,(-1)^{i+1}\,v_{i,\eps}(s)}{\sqrt{4\,v_{i,\eps}(s)^2+v_{i,\eps}'(s)^2}}\,T,
\]
which coincides with the announced expression by virtue of \eqref{eq:Vieps}. On the other hand, the first identity in equation \eqref{eq:zder} gives us
\[
\escpr{B(Z_i),S_i}=|(N_i)_h|^{-1}\,Z_i\big(\escpr{N_i,T}\big)+1,
\] 
so that, after a straightforward computation, the desired formula for $\escpr{B(Z_i),S_i}$ follows from the previous ones for $\escpr{N_i,T}$ and $|(N_i)_h|$.

Now, by using the continuity with respect to $(\eps,s)\in I\times[0,s_i]$ of $v_{i,\eps}(s)$ and $v'_{i,\eps}(s)$, it follows from the expression of $S_i$ that
\[
\lim_{\eps\to\eps_0,\,s\to 0}S_i(\eps,s)=(-1)^{i-1}\,X_{i,\eps_0}(0)=(-1)^{i-1}\,\dot{\Ga}(\eps_0).
\]
Hence $S_i$ extends continuously to $\Ga$ in such a way that $S_2=-S_1=-\dot{\Ga}$. This proves (v).

Finally we show that (vi) holds. From \eqref{eq:Vieps} we get
\[
v_{i,\eps}'(s)=2\,\escpr{X_{i,\eps}(s),J\big(\dot{\ga}_{i,\eps}(s)\big)},
\]
so that the derivative $v_{i,\eps}''(s)$ is continuous with respect to $(\eps,s)$ in $I\times[0,s_i]$. Thus, the fact that $\escpr{B(Z_i),S_i}(\eps,s)\to -1$ when $\eps\to\eps_0$ and $s\to 0$ follows from (i) and the expression for $\escpr{B(Z_i),S_i}$ in (iv). On the other hand, by \eqref{eq:mc2} we obtain
\[
B(Z_i)+S_i=\escpr{B(Z_i),Z_i}\,Z_i+\big(1+\escpr{B(Z_i),S_i}\big)\,S_i=2\la\,|(N_i)_h|\,Z_i+\big(1+\escpr{B(Z_i),S_i}\big)\,S_i,
\]
and so
\begin{equation}
\label{eq:mnhq}
|(N_i)_h|^{-1}\,q_i=|(N_i)_h|^{-1}\,\big(1+\escpr{B(Z_i),S_i}\big)^2+4\,(\la^2+K-1)\,|(N_i)_h|.
\end{equation}
From the expressions for $|(N_i)_h|$ and $\escpr{B(Z_i),S_i}$ in (iv), we have
\[
|(N_i)_h|^{-1}\,\big(1+\escpr{B(Z_i),S_i}\big)^2(\eps,s)=\frac{2\,(-1)^i\,v_{i,\eps}(s)\,\big(v_{i,\eps}''(s)+4\,v_{i,\eps}(s)\big)^2}{\big(4\,v_{i,\eps}(s)^2+v'_{i,\eps}(s)^2\big)^{3/2}},
\]
which tends to $0$ when $\eps\to\eps_0$ and $s\to 0$. This completes the proof.
\end{proof}

We are now ready to introduce a definition. Let $\Ga:I\to M$ be a $C^3$ horizontal curve parameterized by arc-length. For fixed numbers $\la\in\rr$ and $\sg>0$ we say that the set
\begin{equation}
\label{eq:tn}
E_{\la,\sg}:=\Sg_{1,\la}(\Ga)\cup\Sg_{2,\la}(\Ga)
\end{equation} 
is a \emph{$\la$-neighborhood of $\Ga$ of radius $\sg$} if these conditions hold:
\begin{itemize}
\item[(i)] the sets $\Sg_{i,\la}(\Ga)$ defined in \eqref{eq:sigmaila} for $s_1=s_2=\sg$ are well-defined immersed surfaces,
\item[(ii)] the function $v_{i,\eps}$ does not vanish in $(0,\sg]$ for any $i=1,2$. 
\end{itemize}  
Note that the second property implies that the singular set of $E_{\la,\sg}$ equals $\Ga$. 

The following lemma shows a computation that will be useful in the sequel.

\begin{lemma}
\label{lem:conormal}
In a $\la$-neighborhood $E_{\la,\sg}$ of $\Ga$, the outer conormal vector along $\ptl E_{\la,\sg}$ is given by
\[
\nu=\frac{1}{\sqrt{1+\la^2\,\mnh^2}}\,\big(Z-\la\,\mnh\,S\big),
\]
with respect to the Riemannian unit normal $N$ such that $N=T$ along $\Ga$.
\end{lemma}

\begin{proof}
The boundary $\ptl E_{\la,\sg}$ consists of the curves $\beta_i(\eps):=F_i(\eps,\sg)$ with $i=1,2$. Thanks to the formula for $S_i$ in Lemma~\ref{lem:coor} (iv), the tangent vector along these curves is
\[
\dot{\beta}_i(\eps)=X_{i,\eps}(\sg)=\frac{\sqrt{4\,v_{i,\eps}(\sg)^2+v_{i,\eps}'(\sg)^2}}{2\,(-1)^{i-1}}\,\big(\la\,|(N_i)_h|(\eps,\sg)\,Z_i(\eps,\sg)+S_i(\eps,\sg)\big).
\]
Hence, the vector field $Z-\la\,\mnh\,S$ is tangent to $E_{\la,\sg}$ and normal to $\beta_i$. Moreover, by the definition of $\Sg_{i,\la}(\Ga)$ and Lemma~\ref{lem:coor} (iii), it is clear that this vector points outside $E_{\la,\sg}$.
\end{proof}

\subsection{The surfaces $\cmula$ in $3$-dimensional space forms}
\label{subsec:cmula}
\noindent

Here we analyze complete volume-preserving area-stationary $C^2$ surfaces with singular curves in $3$-dimensional space forms. These surfaces can be geometrically described as follows.

Let $M$ be a $3$-dimensional space form of Webster curvature $\kappa$. Take a complete CC-geodesic $\Ga:\rr\to M$ of curvature $\mu$, and fix $\la\in\rr$ such that $\la^2+\kappa>0$. Following the notation preceding Lemma~\ref{lem:coor} we consider the maps $F_{i}(\eps,s):=\ga_{i,\eps}(s)$, the vector fields $X_{i,\eps}:=\ptl F_i/\ptl\eps$ and the functions $v_{i,\eps}:=\escpr{X_{i,\eps},T}$. By using Lemma~\ref{lem:ccjacobi} (iii) together with identities $v_{i,\eps}(0)=0$, $v'_{i,\eps}(0)=2\,(-1)^i$ and $v''_{i,\eps}(0)=-4\mu$, we get that $v_{i,\eps}$ does not depend on $\eps\in\rr$, and it is given by
\begin{equation}
\label{eq:vertimodel}
v_{i,\eps}(s)=v_i(s):=\frac{2}{\sqrt{\tau}}\left\{\frac{-2\mu}{\sqrt{\tau}}\,\big(1-\cos(\sqrt{\tau}s)\big)+(-1)^{i}\,\sin(\sqrt{\tau}s)\right\},
\end{equation}
where $\tau:=4\,(\la^2+\kappa)$. In particular, for any $i=1,2$, there is a \emph{cut constant} $s_i\in (0,2\pi/\sqrt{\tau})$ such that $v_i(s_i)=0$ and $(-1)^{i}\,v_i>0$ on $(0,s_i)$. From the computations in Lemma~\ref{lem:coor}, see also \cite[Lem.~5.7]{hr3}, the associated sets $\Sg_{i,\la}(\Ga)$ in \eqref{eq:sigmaila} satisfy these properties:
\begin{itemize}
\item[(i)] $\Sg_{i,\la}(\Ga)$ is a $C^\infty$ surface immersed in $M$,
\item[(ii)] $\Sg_{i,\la}(\Ga)$ has constant mean curvature $H=\la$ with respect to the unit normal $N_i$ such that $N_i=T$ along $\Ga$. Any CC-geodesic $\ga_{i,\eps}(s)$ with $s\in(0,s_i)$ is a characteristic curve of $\Sg_{i,\la}(\Ga)$,
\item[(iii)] the singular set of $\Sg_{i,\la}(\Ga)$ is parameterized by $\Ga(\eps)$ and $\Ga_{i}(\eps):=F_i(\eps,s_i)$, which are CC-geodesics of curvature $\mu$. Moreover, $N_i=-T$ along $\Ga_i$,
\item[(iv)] the curves $\ga_{i,\eps}$ meet orthogonally $\Ga_{i}$ since $\dot{\Ga}_i(\eps)=(-1)^{i-1}\,J(\dot{\ga}_{i,\eps}(s_i))$. 
\end{itemize}
We write $\Ga_1=\Ga_2$ to indicate that $\Ga_{1}$ and $\Ga_{2}$ parameterize the same curve. In that case $\Sg_{1,\la}(\Ga)\cup\Sg_{2,\la}(\Ga)$ is already a complete surface with empty boundary. Otherwise, we continue the construction by means of the surfaces $\Sg_{2,-\la}(\Ga_1)$ and $\Sg_{1,-\la}(\Ga_2)$, which provides two more singular curves $\Ga_3$ and $\Ga_4$. In general, we proceed by induction so that, at step $k+1$, we leave orthogonally from the recently obtained singular curves by CC-geodesics of curvature $(-1)^{k}\,\la$ until we meet other singular curves. We denote by $\cmula$ the union of all the surfaces $\Sg_{i,\pm\la}(\Ga_j)$ obtained in this way. 

Similarly, we may define the surfaces $\cmula$ when $\la^2+\kappa\leq 0$. We are only interested in the particular case of the Heisenberg group $\mm(0)$. It is known that, when $\mu=0$, the curve $\Ga$ is a horizontal line and $\mathcal{C}_0(\Ga)$ is congruent to the hyperboloid paraboloid $t=xy$. If $\mu\neq0$ then $\Ga$ parameterizes a helix and $\mathcal{C}_0(\Ga)$ is congruent to a left-handed minimal helicoid, see \cite[Sect.~6]{rr2}. 

\begin{remark}
\label{re:cosillas}
Examples in $\mm(0)$ and $\mm(1)$ show that the surfaces $\cmula$ need not be compact nor embedded, see \cite[Ex.~6.7]{rr2} and \cite[Ex.~5.8]{hr1}. When $\Ga$ is a CC-geodesic of curvature $\mu=0$ in $\mm(0)$ or $\mm(1)$ we get a one-parameter family $\cmula$ of embedded cylinders or tori with two singular curves, see \cite[Ex.~6.6]{rr2} and \cite[Ex.~5.7]{hr1}. In $\mm(1)$ the surface $\cmula$ is $C^2$ around the singular set only when $\mu/\sqrt{1+\mu^2}\in\mathbb{Q}$. This condition is equivalent to that $\Ga$ is a circle~\cite[Prop.~3.3]{hr1}.
\end{remark}

After previous classification theorems in $\mm(0)$ and $\mm(1)$, see \cite[Thm.~6.11, Thm.~6.15]{rr2} and \cite[Thm.~5.9]{hr1}, the authors established in \cite[Thm.~4.10, Thm.~4.13]{hr2} the following uniqueness result.

\begin{theorem}
\label{th:cmula}
Let $\Sg$ be a complete, connected and oriented volume-preserving area-stationary $C^2$ surface immersed in a $3$-dimensional space form $M$. If $\Sg$ contains a singular curve $\Ga$, then $\Ga$ can be parameterized as a complete CC-geodesic in $M$ and $\Sg=\cmula$ for some $\la\in\rr$.
\end{theorem}

In the remainder of this section we will deduce some facts about the surfaces $\cmula$ to be used in the proof of Theorem~\ref{th:main}. The next lemma comes easily from equation \eqref{eq:vertimodel} and Lemma~\ref{lem:coor} (iv).

\begin{lemma}
\label{lem:corbero}
Along any surface $\Sg_{i,\la}(\Ga)$ with $i=1,2$ we have:
\begin{itemize}
\item[(i)] $|(N_i)_h|$, $\escpr{N_i,T}$, $\escpr{B(Z_i),S_i}$ and $q_i:=|B(Z_i)+S_i|^2+4\,(\kappa-1)\,|(N_i)_h|^2$ only depend on $s$,
\item[(ii)] $v_i(s_i-s)=v_i(s)$, for any $s\in[0,s_i]$,
\item[(iii)] $\escpr{N_i,T}(s_i-s)=-\escpr{N_i,T}(s)$, for any $s\in[0,s_i]$.
\end{itemize}
\end{lemma}

Finally, for an embedded surface $\cmula$, we analyze the injectivity of the maps $F_i(\eps,s)=\ga_{i,\eps}(\eps,s)$ and the intersection between certain pieces of $\cmula$.

\begin{lemma}
\label{lem:embedded}
Suppose that $\cmula$ is embedded. Then, we have:
\begin{itemize}
\item[(i)] the singular curves are simultaneously injective curves or circles of the same length $\ell$,
\item[(ii)] if $F_i(\eps,s)=F_i(\eps',s')$, where $(\eps,s),(\eps',s')\in\rr\times [0,s_i]$ and $i=1,2$, then $(\eps,s)=(\eps',s')$ when $\Ga$ is injective, or $(\eps,s)=(\eps'+m\ell,s')$ for some $m\in\mathbb{Z}$ when $\Ga$ is a circle of length $\ell$.
\end{itemize}
Consider the pieces of $\cmula$ given by $\Sg_i:=\Sg_{i,\la}(\Ga)$ for any $i=1,2$ and $\Sg_3:=\Sg_{1,-\la}(\Ga_2)$. We get:
\begin{itemize}
\item[(iii)] if $\Ga_1=\Ga_2$ then $\Sg_{1}\cap\Sg_{2}=\Ga\cup\Ga_{1}$,
\item[(iv)] if $\Ga_1\neq\Ga_2$ then $\Sg_{1}\cap\Sg_{2}=\Ga$, $\Sg_2\cap\Sg_3=\Ga_2$ and $\Sg_{1}\cap\Sg_{3}=\emptyset$.
\end{itemize}
\end{lemma}

\begin{proof}
We will denote $\Sg:=\cmula$. As $\Sg$ is embedded we can consider the unit normal $N$ defined over $\Sg$ such that $N=T$ along $\Ga$. Clearly $N\circ F_i=N_i$, where $N_i$ is given in \eqref{eq:normalcoor}. Hence $N=-T$ along $\Ga_i$, and so $\Ga\cap\Ga_i=\emptyset$.

We first see that $\Ga$ is either an injective curve or a circle. Suppose that there are $\eps,\eps'\in\rr$ with $\eps<\eps'$ and $\Ga(\eps)=\Ga(\eps')=p$. Note that $\dot{\Ga}(\eps)=\dot{\Ga}(\eps')$: otherwise, we would contradict that $\Ga$ is a $C^1$ curve that locally separates $\Sg$ into two connected components around $p$, see Theorem~\ref{th:structure} (iii). By the uniqueness of CC-geodesics we get $\Ga(t)=\Ga(t+\eps'-\eps)$ for any $t\in\rr$, so that $\Ga$ is a periodic curve. If $\ell$ is the period of $\Ga$, then the previous reasoning shows that $\Ga:[0,\ell)\to M$ is injective. Thus $\Ga$ is a circle of length $\ell$, as we claimed.

To prove (i) we must check that the behaviour of all the singular curves coincides with that of $\Ga$. By construction, it suffices to see this for $\Ga_i$ with $i=1,2$. If $\Ga_i(\eps)=\Ga_i(\eps')$ for some $\eps,\eps'\in\rr$, then the same argument as above yields $\dot{\Ga}_i(\eps)=\dot{\Ga}_i(\eps')$. From the expression of $\dot{\Ga}_i$ and the uniqueness of CC-geodesics we deduce that $\ga_{i,\eps}=\ga_{i,\eps'}$ and so, $\Ga(\eps)=\Ga(\eps')$. Thus $\eps=\eps'$ if $\Ga$ is injective, or $\eps'=\eps+m\ell$ if $\Ga$ is a circle of length $\ell$. This leads to the desired conclusion. 

Now we prove (ii). Let $(\eps,s),(\eps',s')\in\rr\times[0,s_i]$ such that $s\leq s'$ and $F_i(\eps,s)=F_i(\eps',s')=p$. Having in mind that $\Ga\cap\Ga_i=\emptyset$ together with statement (i) and the fact that $\Sg_i-(\Ga\cup\Ga_i)\sub\Sg-\Sg_0$, we can restrict to the case $s,s'\in (0,s_i)$. Since $\ga_{i,\eps}(s)$ with $s\in (0,s_i)$ is a characteristic curve of $\Sg$ we infer that $\dot{\ga}_{i,\eps}(s)=\dot{\ga}_{i,\eps'}(s')=Z_p$. Again, the uniqueness of CC-geodesics implies that $\ga_{i,\eps}(t+s)=\ga_{i,\eps'}(t+s')$ for any $t\in\rr$. In particular $\Ga(\eps)=\ga_{i,\eps'}(s'-s)$, so that $p':=\ga_{i,\eps'}(s'-s)$ is a singular point of $\Sg_i$ such that $N_{p'}=T_{p'}$. From here we obtain $s'=s$, and so $\ga_{i,\eps}=\ga_{i,\eps'}$. As a consequence $\Ga(\eps)=\Ga(\eps')$ and we finish the proof of (ii) by using (i).

Let us prove (iii). It is clear that $\Ga\subeq\Sg_1\cap\Sg_2$. Moreover $\Ga_1\subeq\Sg_1\cap\Sg_2$ when $\Ga_1=\Ga_2$. Take a point $p\in\Sg_1\cap\Sg_2$. We can write $p=F_1(\eps,s)=F_2(\eps',s')$ with $(\eps,s)\in\rr\times[0,s_1]$ and $(\eps',s')\in\rr\times[0,s_2]$. If $s\in (0,s_1)$ and $s'\in (0,s_2)$ then we would reason as in the proof of (ii) to deduce $\ga_{1,\eps}=\ga_{2,\eps'}$. This would give us $\Ga(\eps)=\Ga(\eps')$ and $J(\dot{\Ga}(\eps))=-J(\dot{\Ga}(\eps'))$, which contradicts that $\Ga$ is an injective curve or a circle. Since $s\in\{0,s_1\}$ and $s'\in\{0,s_2\}$ then $p\in\Ga$ or $p\in\Ga_1\cap\Ga_2$. In the latter case, it follows from Theorem~\ref{th:structure} (iii) that $\Ga_1$ and $\Ga_2$ locally coincide around $p$. Since both curves are CC-geodesics of the same curvature we conclude that $\Ga_1=\Ga_2$. This shows (iii) and the first equality in (iv).

Finally, suppose that $\Ga_1\neq\Ga_2$. The singular set of $\Sg_3$ consists of two CC-geodesics $\Ga_2$ and $\Ga_3$ with $N=-T$ along $\Ga_2$ and $N=T$ along $\Ga_3$. From the definition of $\Sg_{i,\la}(\Ga)$ it is easy to get $\Sg_3=\Sg_{1,\la}(\Ga_3)$ and $\Sg_2=\Sg_{2,-\la}(\Ga_2)$. Observe that $\Ga_3\neq\Ga$; on the contrary, we would have $\Sg_3=\Sg_{1,\la}(\Ga_3)=\Sg_{1,\la}(\Ga)=\Sg_1$ and this would give $\Ga_1=\Ga_2$, a contradiction. Then, since $\Sg_3=\Sg_{1,-\la}(\Ga_2)$,  $\Sg_2=\Sg_{2,-\la}(\Ga_2)$ and $\Ga_3\neq\Ga$, we can infer that $\Sg_3\cap\Sg_2=\Ga_2$. 

It remains to prove that $\Sg_1\cap\Sg_3=\emptyset$. Define $G(\eps,s):=\delta_{\eps}(s)$, where $\delta_{\eps}:\rr\to M$ is the CC-geodesic of curvature $-\la$ with $\delta_{\eps}(0)=\Ga_2(\eps)$ and $\dot{\delta}_\eps(0)=J(\dot{\Ga}_2(\eps))$. According to \eqref{eq:vertimodel} the vertical components of the vector fields $\ptl F_{1}/\ptl\eps$ and $\ptl G/\ptl\eps$ are the same function $v(s)$. Thus, the associated cut constants coincide. The map
\begin{equation}
\label{eq:normal3}
N_G(\eps,s):=\frac{-2\,v(s)\,J(\dot{\delta}_{\eps}(s))+v'(s)\,T}{\sqrt{4\,v(s)^2+v'(s)^2}}
\end{equation}
assigns to any $(\eps,s)\in\rr\times[0,s_1]$ a unit normal to $\Sg$ at $G(\eps,s)$ with $N_G(\eps,0)=-T$. Therefore, equality $N\circ G=N_G$ holds. By Lemma~\ref{lem:coor} (iii) the tangent vector $\dot{\delta}_\eps(s)$ coincides with $-Z_{\delta_\eps(s)}$. 

Suppose that there is $p\in\Sg_1\cap\Sg_3$. We write $p=F_1(\eps,s)=G(\eps',s')$ for $(\eps,s),(\eps',s')\in\rr\times[0,s_1]$. After discussing trivial cases we may assume $s,s'\in (0,s_1)$. From the argument in the proof of (ii) we obtain $\gamma_{1,\eps}(t+s)=\delta_{\eps'}(s'-t)$ for any $t\in\rr$. By evaluating at $t=s'$ we get $\ga_{1,\eps}(s+s')=\Ga_2(\eps')$. If we showed that $s+s'=s_1$ then we would have $\Ga_1=\Ga_2$, which is the desired contradiction. Observe that $N_p=N_1(\eps,s)=N_G(\eps',s')$. Then, equations \eqref{eq:normalcoor} and \eqref{eq:normal3} together with equality $\dot{\ga}_{1,\eps}(s)=-\dot{\delta}_{\eps'}(s')=Z_p$ allow us to deduce
\[
\frac{v'(s')}{v(s')}=-\frac{v'(s)}{v(s)}=\frac{v'(s_1-s)}{v(s_1-s)},
\]
where the last equality comes from the symmetry property of $v$ in Lemma~\ref{lem:corbero} (ii). The conclusion $s+s'=s_1$ now follows from the fact that $v'/v$ is decreasing in $(0,s_1)$. Indeed, an easy computation using \eqref{eq:vertimodel} yields 
\[
\frac{a^2\,v^2\,(v'/v)'}{4}=b\,\big(a\,\sin(as)+b\,\cos(as)\big)-(a^2+b^2),
\]
where $a:=\sqrt{\tau}=2\,\sqrt{\la^2+\kappa}$ and $b=-2\mu$. The right hand side above is negative because $a^2>0$ and $|a\,\sin(as)+b\,\cos(as)|\leq\sqrt{a^2+b^2}$. This completes the proof of the lemma.
\end{proof}

\section{Stability inequality for variations moving singular curves}
\label{sec:stability}
\setcounter{equation}{0} 

Let $M$ be a Sasakian sub-Riemannian $3$-manifold. A volume-preserving area-stationary $C^2$ surface $\Sg$ immersed in $M$ is \emph{stable} if $A''(0)\geq 0$ for any volume-preserving variation. In this context we allow the variations to be of class $C^1$, provided the derivative $A''(0)$ exists.

Our main aim in this section is to prove Theorem~\ref{th:staineq}, where we obtain an analytic inequality for stable surfaces under certain deformations moving finitely many singular curves. This inequality was previously shown by Ritor\'e and the authors~\cite[Prop.~5.2]{hrr} for left-handed minimal helicoids in the Heisenberg group $\mm(0)$. Later, Galli \cite[Thm.~8.6]{galli} extended the inequality to second order minima of the area under compactly supported variations in pseudo-Hermitian $3$-manifolds. In the proof of Theorem~\ref{th:staineq} we employ similar arguments to deal with volume-preserving variations of a CMC surface that may be minimal or not.

\subsection{Second variation formulas}
\label{subsec:2ndformulas}
\noindent

To prove the stability inequality in Theorem~\ref{th:staineq} we need to compute second derivatives involving the area and volume functionals for certain variations of a volume-preserving area-stationary surface. Analogous computations have been derived in different settings, see the Introduction of \cite{hr2} for a very complete a list of references.

In \cite[App.~A]{hr2} the authors established a second variation formula for arbitrary admissible variations, possibly moving the singular set, of a CMC surface with boundary. The notion of \emph{admissible variation} was introduced in \cite[Def.~5.1]{hr2}, and it gathers the conditions necessary to differentiate two times under the integral sign in \eqref{eq:duis}. For variations supported on the regular set of the surface the mentioned formula leads to the next result.

\begin{proposition}
\label{prop:2ndvar1} 
Let $\Sg$ be an oriented $C^2$ surface, possibly with boundary, immersed in a Sasakian sub-Riemannian $3$-manifold $M$. Suppose that $\Sg-\Sg_0$ is $C^3$ and has constant mean curvature $H$. If $\varphi:I\times\Sg\to M$ is any $C^3$ variation supported on $\Sg-\Sg_0$ and with velocity vector field $U=f\,N+k\,T$ for some functions $f,k\in C^2_0(\Sg-\Sg_0)$, then the derivative $(A+2H\,V)''(0)$ exists, and we have
\begin{align*}
(A+2H\,V)''(0)&=\int_{\Sg}|N_{h}|^{-1}\left\{Z(u)^2-
\big(|B(Z)+S|^2+4\,(K-1)\,|N_{h}|^2\big)\,u^2\right\}da
\\
\nonumber
&+\int_\Sg\divv_\Sg\big(\mnh\,W^\top\big)\,da+\int_{\Sg}\divv_\Sg\big(\zeta\,
Z+\xi\,S\big)\,da.
\end{align*}
Here $\{Z,S\}$ is the orthonormal basis defined in \eqref{eq:nuh} and \eqref{eq:ese}, the function $u:=f+\escpr{N,T}\,k$ is the normal component of the velocity, $B$ is the Riemannian shape operator, $K$ is the Webster  curvature, $W$ is the acceleration vector field, and the functions $\zeta,\xi$ are defined by
\begin{align*}
\zeta&:=\escpr{N,T}\,\big(1-\escpr{B(Z),S}\big)\,u^2+\mnh^2\,\left\{\escpr{N,T}\,\big(1-\escpr{B(Z),S}\big)\,k^2-2\,\escpr{B(Z),S}\,f\,k\right\},
\\
\xi&:=\escpr{N,T}\,\big(2H\,\mnh\,u^2-\escpr{W,N}\big)-2H\,\escpr{N,T}^3\,\mnh\,k^2+2H\,\mnh\,\big(1-2\,\escpr{N,T}^2\big) f\,k.
\end{align*}
\end{proposition}

\begin{proof}
Following \cite[Ex.~8.2]{hr2}, for any $C^3$ variation $\var:I\times\Sg\to M$ supported on $\Sg-\Sg_0$, there is a open interval $I'\subset\subset I$ with $0\in I'$ such that the restriction of $\var$ to $I'\times\Sg$ is admissible in the sense of \cite[Def.~5.1]{hr2}. Thus, we can apply \cite[Thm.~7.1]{hr2} to deduce that the derivative $(A+2H\,V)''(0)$ is finite, and can be computed as
\begin{align*}
(A+2H\,V)''(0)&=\int_{\Sg}|N_{h}|^{-1}\left\{Z(u)^2-
\big(|B(Z)+S|^2+4\,(K-1)\,|N_{h}|^2\big)\,u^2\right\}da
\\
\nonumber
&+\int_{\Sg}\divv_\Sg\big\{\escpr{N,T}\,\big(1-\escpr{B(Z),S}\big)\,u^2\,Z\big\}\,da
\\
\nonumber
&+\int_{\Sg}\divv_\Sg\big\{\escpr{N,T}\,\big(2H\,\mnh\,u^2-\escpr{W,N}\big)\,S\big\}\,da
\\
\nonumber
&+\int_\Sg\divv_\Sg\big(\mnh\,W^\top\big)\,da+\int_{\Sg}\divv_\Sg\big(h_1
Z+h_2\,S\big)\,da,
\end{align*}
where $h_1$ and $h_2$ are the functions involving the tangent vector field $Q:=U^\top$ given by
\begin{align*}
h_1:=2&\,\big\{H\,\escpr{Q,Z}+\escpr{N,T}\,\escpr{D_SQ,Z}+\mnh^{-1}\,\escpr{Q,S}\,\big(\escpr{B(Z),S}+\escpr{N,T}^2\big)\big\}\,u
\\
\nonumber
&+\mnh\,\big(\escpr{Q,Z}\,\escpr{D_SQ,S}-\escpr{Q,S}\,\escpr{D_SQ,Z}\big)
\\
\nonumber
&+\escpr{N,T}\,\escpr{Q,Z}^2\,\big(1-\escpr{B(Z),S}\big)-\escpr{N,T}\,\escpr{Q,Z}\,\escpr{Q,S}\,\escpr{B(S),S},
\end{align*}
\vspace{-0,5cm}
\begin{align*}
\hspace{-0,67cm}h_2:=-2&\,\big\{H\,\escpr{Q,S}+\escpr{N,T}\,\escpr{D_ZQ,Z}-\mnh\,\escpr{Q,Z}\big\}\,u
\\
\nonumber
&+\mnh\,\big(\escpr{Q,S}\,\escpr{D_ZQ,Z}-\escpr{Q,Z}\,\escpr{D_ZQ,S}\big)
\\
\nonumber
&+2H\,\mnh\,\escpr{N,T}\,\escpr{Q,Z}^2+\escpr{N,T}\,\escpr{Q,Z}\,\escpr{Q,S}\,\big(1+\escpr{B(Z),S}\big).
\end{align*}
Hence, in order to prove the proposition, it suffices to see that
\begin{equation}
\label{eq:h1h2}
\begin{split}
h_1&=\mnh^2\,\left\{\escpr{N,T}\,\big(1-\escpr{B(Z),S}\big)\,k^2-2\,\escpr{B(Z),S}\,f\,k\right\},
\\
h_2&=-2H\,\escpr{N,T}^3\,\mnh\,k^2+2H\,\mnh\,\big(1-2\,\escpr{N,T}^2\big) f\,k.
\end{split}
\end{equation}
This requires some calculus. On the one hand note that $Q=-(\mnh\,k)\,S$ by \eqref{eq:relations}, and so
\begin{equation}
\label{eq:qzqs}
\escpr{Q,Z}=0, \quad \escpr{Q,S}=-\mnh\,k.
\end{equation} 
On the other hand, by using the first equality in \eqref{eq:divzs} and \eqref{eq:dzz} we get
\[
\mnh^{-1}\,\escpr{N,T}\,\big(1+\escpr{B(Z),S}\big)=\divv_\Sg Z=\escpr{D_ZZ,Z}+\escpr{D_SZ,S}=\escpr{D_SZ,S},
\]
so that
\begin{equation}
\label{eq:dsqz}
\escpr{D_SQ,Z}=-\escpr{Q,D_SZ}=\mnh\,k\,\escpr{S,D_SZ}=\escpr{N,T}\,\big(1+\escpr{B(Z),S}\big)\,k.
\end{equation}
Again by equation \eqref{eq:dzz} we infer
\begin{equation}
\label{eq:dzqz}
\escpr{D_ZQ,Z}=-\escpr{Q,D_ZZ}=\mnh\,k\,\escpr{S,2H\,\nuh}=2H\,\mnh\,\escpr{N,T}\,k.
\end{equation}
Now, we substitute \eqref{eq:qzqs}, \eqref{eq:dsqz}, \eqref{eq:dzqz} and equality $u=f+\escpr{N,T}\,k$ into the definitions of $h_1$ and $h_2$ above. After simplifying with the help of the identity $\mnh^2+\escpr{N,T}^2=1$ we obtain \eqref{eq:h1h2}.
\end{proof}

\begin{remark} 
\label{re:lessreg}
The proposition still holds under weaker regularity assumptions on $\var$. For instance, it is valid for a variation $\var:I\times\Sg\to M$ of the form
\[ 
\var_r(p):=\exp_p(r\,U_p+t(r)\,w(p)\,N_p),
\] 
where $\exp_p$ stands for the Riemannian exponential map at $p$, the vector field $U$ is $C^2$ with compact support on $\Sg-\Sg_0$, and we have functions $t\in C^2(I)$ and $w\in C^2_0(\Sg-\Sg_0)$ with $t(0)=0$. In this situation, arguments similar to those in the proof of \cite[Thm.~5.2]{rosales} provide enough regularity to derive the second variation formula as in the $C^3$ case. For minimal surfaces in the Heisenberg group $\mm(0)$ our formula recovers the one in \cite[Thm.~3.7]{hrr}. A similar formula for minimal surfaces in pseudo-Hermitian $3$-manifolds was given in \cite[Thm.~7.3]{galli}. We emphasize that Proposition~\ref{prop:2ndvar1} is true \emph{for any variation} $\var$ with velocity $U:=f\,N+k\,h$, and for CMC surfaces with $H\neq 0$.
\end{remark}

Next, we compute the second derivative of volume for vertical deformations of $\Sg$. By a \emph{vertical variation} of $\Sg$ we mean a variation of the form $\var_r(p):=\exp_p(r\,\rho(p)\,T_p)$, where $\rho$ has compact support on $\Sg$ and $\exp_p$ denotes the exponential in $(M,g)$ at the point $p$. 

\begin{lemma}
\label{lem:2ndvolvert}
Let $\Sg$ be an oriented $C^2$ surface immersed in a Sasakian sub-Riemannian $3$-manifold. Then, for any function $\rho\in C^1_0(\Sg)$, the volume functional $V(r)$ associated to the $C^1$ vertical variation $\var_r(p):=\exp_p(r\,\rho(p)\,T_p)$ satisfies $V''(0)=0$.
\end{lemma}

\begin{proof}
The velocity of the variation is the vector field $U=\rho\,T$ having as normal component the function $u:=\escpr{N,T}\,\rho$. Moreover, the acceleration $W$ vanishes since, for any $p\in\Sg$, the curve $\ga_p(r):=\var_r(p)$ is a Riemannian geodesic. We can proceed as in the proof of \cite[Eq.~(5.18)]{rosales} and \cite[Eq.~(7.15)]{hr2} to compute the second derivative of volume. We obtain
\begin{align*}
V''(0)=\int_{\Sg}\left\{\escpr{W,N}-\escpr{\nabla_\Sg u,Q}-\escpr{B(Q),Q}+u\,\divv_\Sg Q-\big(2H\,\mnh+\escpr{B(S),S}\big)\,u^2\right\}da,
\end{align*}
where $\nabla_\Sg$ denotes the gradient relative to $\Sg$, the function $H$ is the mean curvature of $\Sg$, and $Q$ is the tangent projection of $U$. Note that $Q=-(\mnh\,\rho)\,S$ by the second equality in \eqref{eq:relations}. By using equations \eqref{eq:sder} and \eqref{eq:divzs} we get
\begin{align*}
\escpr{\nabla_\Sg u,Q}&=-(\mnh\,\rho)\,S\big(\escpr{N,T}\,\rho\big)=-\mnh^2\,\escpr{B(S),S}\,\rho^2-\mnh\,\escpr{N,T}\,S(\rho)\,\rho,
\\
\escpr{B(Q),Q}&=\mnh^2\,\escpr{B(S),S}\,\rho^2,
\\
\divv_\Sg Q&=-\divv_\Sg\big(\mnh\,\rho\,S\big)=-\mnh\,\rho\,\divv_\Sg S-S\big(\mnh\,\rho\big)
\\
&=2H\,\mnh\,\escpr{N,T}\,\rho+\escpr{N,T}\,\escpr{B(S),S}\,\rho-\mnh\,S(\rho).
\end{align*}
By substituting these equalities into the expression for $V''(0)$ we see that the integrand vanishes. This completes the proof of the lemma.
\end{proof}

We finally deduce a second variation formula for the area of a vertical variation which is constant along any CC-geodesic segment of a $\la$-neighborhood of some horizontal curve. In \cite[Lem.~8.5]{hr2} the authors showed conditions for a vertical variation around the singular set to be admissible. For such variations we could compute $A''(0)$ from \cite[Thm.~7.1]{hr2}. Unfortunately, this requires the local integrability of $\mnh^{-1}$ with respect to $da$, which is not guaranteed near a singular curve. To avoid this difficulty we will follow the arguments in \cite[Prop.~3.11]{hrr} for area-stationary surfaces in the Heisenberg group $\mm(0)$, see \cite[Lem.~7.7]{galli} for the case of area-stationary surfaces in pseudo-Hermitian $3$-manifolds. 

\begin{proposition}
\label{prop:2ndvar2} 
Let $\Ga:I\to M$ be a $C^3$ horizontal curve parameterized by arc-length in a Sasakian sub-Riemannian $3$-manifold $M$. Consider a $\la$-neighborhood $\Sg=E_{\la,\sg}$ of $\Ga$ as defined in \eqref{eq:tn}. Take $\rho\in C^1(\Sg)$ such that $Z(\rho)=0$ in $\Sg-\Ga$ and $\rho_\Ga:=\rho\circ\Ga$ is a $C^2$ function with compact support. Then, for any $\la$-neighborhood $E_0\subeq\Sg$ of $\Ga$ with small enough radius, the area functional $A(r):=A(\var_r(E_0))$ of the $C^1$ vertical variation $\var_r(p):=\exp_p(r\,\rho(p)\,T_p)$ satisfies
\[
A''(0)=\int_I\rho_\Ga'(\eps)^2\,d\eps=\int_\Ga S(\rho)^2\,dl.
\]
In this formula $dl$ is the length element in $M$, and the value of $S$ along $\Ga$ is any of the continuous extensions in Lemma~\ref{lem:coor} (v).
\end{proposition}

\begin{proof}
Let $[-\eps_0,\eps_0]\sub I$ be an interval containing the support of $\rho_\Ga$. By equation~\eqref{eq:tn} we can write $\Sg=\Sg_1\cup\Sg_2$, where $\Sg_i$ denotes the surface $\Sg_{i,\la}(\Ga)$ defined in \eqref{eq:sigmaila} for $s_1=s_2=\sg$. Choose the unit normal vector $N$ on $\Sg$ whose restriction to $\Sg_i$ equals the vector field $N_i$ in Lemma~\ref{lem:coor} (iii). 

We first prove that the area functional $A(r):=A(\var_r(\Sg))$ satisfies
\begin{equation}
\label{eq:arearep}
A(r)=\int_{-\eps_0}^{\eps_0}\bigg(\int_{-\sg}^{\sg}|r\,\rho_\Ga'(\eps)+v_\eps(s)|\,ds\bigg)\,d\eps,
\end{equation}
where $v_\eps:[-\sg,\sg]\to\rr$ is the function
\[
v_\eps(s):=
\begin{cases}
v_{1,\eps}(s)\,\,\,&\text{if\, $s\in [0,\sg]$},
\\
v_{2,\eps}(-s)\,\,\,&\text{if\, $s\in[-\sg,0]$},
\end{cases}
\]
and $v_{i,\eps}$ is defined just below equation \eqref{eq:sigmaila}. The starting point to show \eqref{eq:arearep} is the equality \eqref{eq:duis}, which tells that
\[
A(r)=\int_{\Sg}\mnh_r(p)\,|\text{Jac}\,\var_r|(p)\,da,
\]
for any $r$ small enough. By following the proof of \cite[Lem.~8.5]{hr2} and having in mind that $Z(\rho)=0$ on $\Sg$, we get 
\[
\mnh_r(p)\,|\text{Jac}\,\var_r|(p)=Q_p(r)^{1/2},
\]
where
\[
Q_p(r):=a_p\,r^2+b_p\,r+c_p,
\]
and the coefficients $a_p$, $b_p$ and $c_p$ are given by
\[
a_p:=S_p(\rho)^2,\quad b_p=-2\,\mnh(p)\,S_p(\rho),\quad c_p:=\mnh^2(p),
\]
for any $p\in\Sg-\Ga$. From here it is easy to check that
\[
A(r)=\int_{\Sg}\big|S(\rho)\,r-\mnh\big|\,da.
\]
Now, we compute the previous integral with respect to the coordinates $(\eps,s)$ associated to $\Sg_i$. Since $Z(\rho)=0$ in $\Sg-\Ga$ then $\rho$ is constant along any CC-geodesic segment $\ga_{i,\eps}$ of curvature $\la$ leaving orthogonally from $\Ga$. In particular $\rho\big(\ga_{i,\eps}(s)\big)=0$ when $\eps\notin[-\eps_0,\eps_0]$ and $s\in(-\sg,\sg)$. From the identities in Lemma~\ref{lem:coor} (iv) and the equality
\[
X_{i,\eps}(\rho)=\frac{d}{d\eps}\bigg|_\eps\rho\big(\ga_{i,\eps}(s)\big) =\frac{d}{d\eps}\bigg|_\eps (\rho\circ\Ga)(\eps)=\rho_\Ga'(\eps),
\]
we deduce
\[
A(r)=\sum_{i=1}^2\int_{[-\eps_0,\eps_0]\times[0,\sg]}|r\,\rho_\Ga'(\eps)+v_{i,\eps}(s)|\,d\eps\,ds=\int_{-\eps_0}^{\eps_0}\bigg(\int_{-\sg}^{\sg}|r\,\rho_\Ga'(\eps)+v_\eps(s)|\,ds\bigg)\,d\eps,
\]
and so equation \eqref{eq:arearep} holds.

Next we derive the announced formula for $A''(0)$. Thanks to Lemma~\ref{lem:coor} (i) and the definition of $v_{i,\eps}$ we see that $v_\eps(s)$ is $C^1$ with respect to $(\eps,s)\in (-\eps_0,\eps_0)\times (-\sg,\sg)$. By equality $v_\eps'(0)=-2$, there is $\sg'\in (0,\sg)$ such that $v_\eps'(s)<0$ for any $\eps\in[-\eps_0,\eps_0]$ and $ s\in (-\sg',\sg')$. We define
\[
f(\eps,s,r):=r\,\rho_\Ga'(\eps)+v_\eps(s).
\] 
This is a $C^1$ function with $f(\eps,0,0)=0$. Moreover $(\ptl f/\ptl s)(\eps,s,r)=v_\eps'(s)$, so that $(\ptl f/\ptl s)(\eps,0,0)=-2$. By applying the implicit function theorem we can find $\sg'_0\in (0,\sg')$, $r_0>0$, and a $C^1$ function $s:(-\eps_0,\eps_0)\times(-r_0,r_0)\to (-\sg'_0,\sg'_0)$ such that $s(\eps,0)=0$ for any $\eps\in (-\eps_0,\eps_0)$, and equality $f(\eps,s,r)=0$ holds for a triple in $R:=(-\eps_0,\eps_0)\times (-\sg'_0,\sg'_0)\times (-r_0,r_0)$ if and only if $s=s(\eps,r)$. For $(\eps,s,r)\in R$, the fact that $(\ptl f/\ptl s)(\eps,s,r)<0$ implies that $f(\eps,s,r)>0$ if $s\in (-\sg'_0,s(\eps,r))$ and $f(\eps,s,r)<0$ if $s\in (s(\eps,r),\sg'_0)$. Let $E_0:=E_{\la,\sg_0}$ be any $\la$-neighborhood of $\Ga$ of radius $\sg_0<\sg_0'$. From \eqref{eq:arearep}, the area functional $A(r):=A(\var_r(E_0))$ equals
\[
A(r)=\int_{-\eps_0}^{\eps_0}\,\psi_\eps(r)\,d\eps,
\]
where
\[
\psi_\eps(r):=\int_{-\sg_0}^{s(\eps,r)}f(\eps,s,r)\,ds-\int_{s(\eps,r)}^{\sg_0}f(\eps,s,r)\,ds.
\]
We differentiate into the definition of $\psi_\eps(r)$ with respect to $r$. By using that $f(\eps,s(\eps,r),r)=0$ and that $(\ptl f/\ptl r)(\eps,s,r)=\rho_\Ga'(\eps)$, we infer
\[
\psi'_\eps(r)=\int_{-\sg_0}^{s(\eps,r)}\frac{\ptl f}{\ptl r}(\eps,s,r)\,ds-\int_{s(\eps,r)}^{\sg_0}\frac{\ptl f}{\ptl r}(\eps,s,r)\,ds=2\,\rho_\Ga'(\eps)\,s(\eps,r),
\]
so that
\[
\psi_\eps''(0)=2\,\rho_\Ga'(\eps)\,\frac{\ptl s}{\ptl r}(\eps,0).
\]
On the other hand, by differentiating with respect to $r$ into the equality
\[
0=f(\eps,s(\eps,r),r)=r\,\rho_\Ga'(\eps)+v_\eps\big(s(\eps,r)\big),
\]
we obtain $(\ptl s/\ptl r)(\eps,0)=\rho_\Ga'(\eps)/2$, so that $\psi_\eps''(0)=\rho_\Ga'(\eps)^2$. Finally we have
\[
A''(0)=\int_{-\eps_0}^{\eps_0}\psi''_\eps(0)\,d\eps=\int_{-\eps_0}^{\eps_0}\rho_\Ga'(\eps)^2\,d\eps,
\]
which completes the proof.
\end{proof}

\subsection{Stability inequality}
\label{subsec:staineq}
\noindent

Before stating the main result of this section we need a lemma, which allows to construct volume-preserving variations of a surface $\Sg$ with prescribed velocity $U$. In the proof we will follow the ideas that Barbosa and do Carmo employed in Euclidean space when $U$ is normal to $\Sg$, see \cite[Lem.~2.4]{bdc}.

\begin{lemma}
\label{lem:bdc}
Let $\Sg$ be an oriented $C^2$ surface immersed in a Sasakian sub-Riemannian $3$-manifold $M$. Take a $C^1$ vector field $U$ with compact support on $\Sg$ and satisfying $\int_\Sg u\,da=0$, where $u:=\escpr{U,N}$. For any function $w\in C^2_0(\Sg)$ with $\int_\Sg w\,da\neq 0$, there is a $C^2$ function $t:(-r_0,r_0)\to\rr$ with $t(0)=t'(0)=0$ such that the map
\[
\var_r(p):=\exp_p\big(r\,U_p+t(r)\,w(p)\,N_p\big)
\]
defines a $C^1$ volume-preserving variation of $\Sg$ with velocity vector field $U$.
\end{lemma}

\begin{proof}
For $r$ and $t$ small enough, the family of maps $\psi_{r,t}:\Sg\to M$ given by
\[
\psi_{r,t}(p):=\exp_p(r\,U_p+t\,w(p)\,N_p)
\]
is a two-parameter $C^1$ variation of $\Sg$. Let $\Sg_{r,t}$ be the associated immersed surface and $V(r,t)$ the signed volume enclosed between $\Sg$ and $\Sg_{r,t}$. By using the first derivative of volume \cite[Lem.~2.1]{bdce} with the one-parameter variations $\psi_{r,0}$ and $\psi_{0,t}$, we have
\[
\frac{\ptl V}{\ptl r}(0,0)=\int_\Sg u\,da=0,\quad\frac{\ptl V}{\ptl t}(0,0)=\int_\Sg w\,da\neq 0.
\] 
Since $V(r,t)$ is a $C^2$ function there exists, by the implicit function theorem, a $C^2$ function $t(r)$ with $r\in(-r_0,r_0)$ such that $t(0)=0$ and $V(r,t(r))=V(0,0)$ for any $r\in (-r_0,r_0)$. By differentiating with respect to $r$, and taking into account the previous equalities, we infer that $t'(0)=0$. Finally, the map $\var_r(p):=\psi_{r,t(r)}(p)$ where $r\in(-r_0,r_0)$ and $p\in\Sg$ provides the announced variation. 
\end{proof}

Now, we are ready to prove the following statement.

\begin{theorem}
\label{th:staineq}
Let $M$ be a Sasakian sub-Riemannian $3$-manifold and $\Sg$ an oriented $C^2$ surface immersed in $M$ with $C^3$ regular set $\Sg-\Sg_0$ and constant mean curvature $H$. Consider any function $u\in C^1_0(\Sg)\cap C^2(\Sg-\Sg_0)$ with $\int_\Sg u\,da=0$, and satisfying that:
\begin{itemize}
\item[(i)] the restriction of $u$ to $\Sg_0$ is a $C^2$ function with compact support in a set $\Lambda=\{\Ga_1,\ldots,\Ga_m\}$ of singular curves of class $C^3$,
\item[(ii)] there are pairwise disjoint sets $\{E_1,\ldots,E_m\}$ in $\Sg$ such that $E_j$ is a $\la_j$-neighborhood of $\Ga_j$, the function $\escpr{N,T}$ never vanishes on $E_j$ and, either $Z(u/\escpr{N,T})=0$ in $E_j-\Ga_j$ for any $j=1,\ldots,m$, or $Z(u)=0$ in $E_j-\Ga_j$ for any $j=1,\ldots,m$.
\end{itemize}
If $\Sg$ is stable, then we have
\begin{equation}
\label{eq:indexform}
0\leq \mathcal{Q}(u):=\int_\Sg\mnh^{-1}\big(Z(u)^2-q\,u^2\big)\,da+\int_{\Lambda}\big(S(u)^2-4\,u^2\big)\,dl,
\end{equation}
where $q:=|B(Z)+S|^2+4\,(K-1)\,\mnh^2$. Here $\{Z,S\}$ is the orthonormal basis in \eqref{eq:nuh} and \eqref{eq:ese}, $B$ is the Riemannian shape operator, $K$ is the Webster curvature of $M$, the value of $S$ along $\Ga_j$ is any of the continuous extensions in Lemma~\ref{lem:coor} (v), and $dl$ denotes the length element in $M$.
\end{theorem}

\begin{remarks}
\label{re:pachina}
(i). The surface $\Sg$ could be volume-preserving area-stationary or not. Anyway, we understand that $\Sg$ is stable if $A''(0)\geq 0$ for any volume-preserving variation. 

(ii). The definition of $\mathcal{Q}$ and the inequality $\mathcal{Q}(u)\geq 0$ do not depend on the orientation of $\Sg$. However, Theorem~\ref{th:structure} and Lemma~\ref{lem:coor} (iii) show that the orientation determines the values of $\la_j$ in such a way that $\la_j=\escpr{N,T}_{|\Ga_j}\,H$. 

(iii). Note that $\mathcal{Q}(u)$ is well defined for any function $u\in C_0(\Sg)\cap C^1_0(\Lambda)$ with support disjoint from the isolated singular points of $\Sg$, and such that $\mnh^{-1}\,Z(u)^2$ is integrable with respect to $da$. We do not need to assume the integrability of $\mnh^{-1}\,q\,u^2$; this comes from Lemma~\ref{lem:coor} (vi) since $\mnh^{-1}\,q$ extends continuously to $\Lambda$. 
\end{remarks}

\begin{proof}[Proof of Theorem~\ref{th:staineq}]
We first observe that $\mnh^{-1}\,Z(u)^2$ is integrable on $\Sg$ for any function $u$ as in the statement. This is obvious if $Z(u)=0$ in $E_j-\Ga_j$ for any $j=1,\ldots,m$. In case $Z(u/\escpr{N,T})=0$ we get $\escpr{N,T}\,Z(u)=Z(\escpr{N,T})\,u$, and from \eqref{eq:zder} we deduce that
\[
\mnh^{-1}\,Z(u)^2=\frac{\mnh}{\escpr{N,T}^2}\,\big(\escpr{B(Z),S}-1\big)^2\,u^2,
\]
which extends continuously to any singular curve $\Ga_j$ by Lemma~\ref{lem:coor} (vi).

To prove the theorem we distinguish two situations.

\vspace{0,1cm}
\emph{Case 1}. Suppose that $Z(u/\escpr{N,T})=0$ in $E_j-\Ga_j$ for any $j=1,\ldots,m$. 

From Proposition~\ref{prop:2ndvar2} we can find $\sg_0>0$ such that, if $E_\sg$ denotes the union of the $\la_j$-neighborhoods of $\Ga_j$ of radius $\sg\leq\sg_0$, then the second derivative of the area for the vertical variation of $E_\sg$ defined by $\psi_r(p):=\exp_p(r\,\rho(p)\,T_p)$ with $\rho:=u/\escpr{N,T}$ is given by
\begin{equation}
\label{eq:spirit0}
\frac{d^2}{dr^2}\bigg|_{r=0}A\big(\psi_r(E_\sg)\big)=\int_{\Lambda}S(\rho)^2\,dl=\int_{\Lambda}S(u)^2\,dl,
\end{equation}
where in the second equality we have used that $\escpr{N,T}=\pm 1$ along the singular curves in $\Lambda$.

Let $C$ be the support of $u$. For any $\sg\in(0,\sg_0/2)$ we take functions $a_\sg,b_\sg\in C^2_0(\Sg)$ such that $b_\sg=1$ on $\overline{E}_\sg\cap C$, the support of $b_\sg$ is contained in $E_{2\sg}$, and $a_\sg+b_\sg=1$ on $C$. We also define a vector field $U_\sg$ on $\Sg$ by
\[
U_\sigma:=(a_\sg\,u)\,N+(b_\sg\,\rho)\,T.
\]
Note that $U_\sg$ is $C^1$ on $\Sg$ and $C^2$ on $\Sg-\Sg_0$. Moreover it has compact support contained in $C$, normal component $\escpr{U_\sg,N}=u$, and satisfies $U_\sg=\rho\,T$ on $E_\sg$. Take any function $w_\sg\in C^2_0(\Sg)$ with support disjoint from $E_\sg$ and $\int_\Sg w_\sg\,da\neq 0$. By applying Lemma~\ref{lem:bdc}, there is a $C^2$ function $t_\sg(r)$ with $t_\sg(0)=t_\sg'(0)=0$, such that the map
\[
\var^\sg(r,p)=\var^\sg_r(p):=\exp_p\big(r\,(U_\sg)_p+t_\sg(r)\,w_\sg(p)\,N_p\big)
\]
defines a $C^1$ volume-preserving variation of $\Sg$ with velocity vector field $U_\sg$. If $A_\sg(r)$ stands for the associated area functional, then the stability of $\Sg$ implies that $A_\sg''(0)\geq 0$. From here we will prove that $\mathcal{Q}(u)\geq 0$ by computing $A''_\sg(0)$ and letting $\sg\to 0$.

The function $A_\sg(r)$ can be written as
\[
A_\sg(r)=A\big(\var^\sg_r(E_\sg)\big)+A\big(\var^\sg_r(\Sg-E_\sg)\big).
\]
Observe that $\var_r^\sg=\psi_r$ on $E_\sg$ because $U_\sg=\rho\,T$ and $w_\sg=0$ on $E_\sg$. Hence, we can compute the second derivative of $A\big(\var^\sg_r(E_\sg)\big)$ from equation~\eqref{eq:spirit0}. On the other hand, since $\var^\sg$ is a volume-preserving variation, we conclude from Lemma~\ref{lem:2ndvolvert} that
\[
\frac{d^2}{dr^2}\bigg|_{r=0}V\big(\var_r^\sg(\Sg-E_\sg)\big)=0.
\]
This allows to compute the second derivative of $A\big(\var^\sg_r(\Sg-E_\sg)\big)$ by means of Proposition~\ref{prop:2ndvar1} and Remark~\ref{re:lessreg}. All this, together with the Riemannian divergence theorem, yields
\[
\begin{split}
A_\sg''(0)&=\frac{d^2}{dr^2}\bigg|_{r=0}A\big(\var^\sg_r(E_\sg)\big)+\frac{d^2}{dr^2}\bigg|_{r=0}A\big(\var^\sg_r(\Sg-E_\sg)\big)
\\
&=\int_{\Sg-E_\sg}|N_{h}|^{-1}\,\big(Z(u)^2-q\,u^2\big)\,da+\int_{\Lambda}S(u)^2\,dl
\\
&-\int_{\ptl E_\sg}\mnh\,\escpr{W_\sg,\nu_\sg}\,dl-\int_{\ptl E_\sg}\zeta_\sg\,\escpr{Z,\nu_\sg}\,dl-\int_{\ptl E_\sg}\xi_\sg\,\escpr{S,\nu_\sg}\,dl,
\end{split}
\]
where $W_\sg$ is the acceleration associated to $\var^\sg$, the notation $\nu_\sg$ stands for the unit conormal along $\ptl E_\sg$ pointing into $\Sg-E_\sg$, and the functions $\zeta_\sg,\xi_\sg\in C^1_0(\Sg-E_\sg)$ are those defined in Proposition~\ref{prop:2ndvar1} for $f_\sg:=a_\sg\,u$ and $k_\sg:=b_\sg\,\rho$. 

To finish the proof in this case we will see that $\lim_{\sg\to 0}A_\sg''(0)=\mathcal{Q}(u)$. By the dominated convergence theorem it is clear that
\begin{equation*}
\lim_{\sg\to 0}\,\int_{\Sg-E_\sg}\mnh^{-1}\,\big(Z(u)^2-q\,u^2\big)\,da=\int_{\Sg}\mnh^{-1}\,\big(Z(u)^2-q\,u^2\big)\,da.
\end{equation*}
Moreover, since $\var_r^\sg=\psi_r$ on $E_\sg$, then we have $\ga_p(r):=\var^\sg_r(p)=\exp_p(r\,\rho(p)\,T_p)$ for any $p\in\ptl E_\sg$. The curve $\ga_p$ is a geodesic in $(M,g)$, so that
\begin{equation*}
(W_\sg)_p=\dot{\ga}'_p(0)=0,\quad p\in \ptl E_\sg.
\end{equation*}
Note also that $f_\sg=0$ and $k_\sg=\rho=u/\escpr{N,T}$ along $\ptl E_\sg$. By the definition of $\xi_\sg$ in the statement of Proposition~\ref{prop:2ndvar1} we get $\xi_\sg=0$ and $\zeta_\sg=\escpr{N,T}^{-1}\,(1-\escpr{B(Z),S})\,u^2$ along $\ptl E_\sg$. Having all this in mind, it remains to check that
\begin{equation}
\label{eq:spirit23}
\lim_{\sg\to 0}\int_{\ptl E_\sg}\frac{1-\escpr{B(Z),S}}{\escpr{N,T}}\,\escpr{Z,\nu_\sg}\,u^2\,dl=4\int_\Lambda u^2\,dl.
\end{equation}

Let $\Ga_j:I\to M$ be a parameterization by arc-length of one of the singular curves in $\Lambda$. Denote by $(E_j)_\sg\subeq E_\sg$ the $\la_j$-neighborhood of $\Ga_j$ of radius $\sg$. By taking into account Lemma~\ref{lem:conormal} and Remarks~\ref{re:pachina} (ii), we infer that
\[
\nu_\sg=\frac{\escpr{N,T}_{|\Ga_j}}{\sqrt{1+H^2\,\mnh^2}}\,\big(Z-H\,\mnh\,S\big)\quad\text{along} \ \ptl (E_j)_\sg.
\]
Thus, we obtain
\[
\int_{\ptl (E_j)_\sg}\frac{1-\escpr{B(Z),S}}{\escpr{N,T}}\,\escpr{Z,\nu_\sg}\,u^2\,dl=\int_{\ptl (E_j)_\sg}\frac{\escpr{N,T}_{|\Ga_j}\big(1-\escpr{B(Z),S}\big)}{\escpr{N,T}\,\sqrt{1+H^2\,\mnh^2}}\,u^2\,dl.
\]
We compute the last integral with respect to the coordinates $(\eps,s)\in I\times[0,\sg]$ over $(E_j)_\sg$ defined in Lemma~\ref{lem:coor}. Along any of the two curves $\Ga_j^i$ in $\ptl (E_j)_\sg$ this integral equals
\[
\frac{\escpr{N,T}_{|\Ga_j}}{2}\,\int_I\frac{1-\escpr{B(Z),S}}{\escpr{N,T}\,\sqrt{1+H^2\,\mnh^2}}\,(\eps,\sg)\,\,u(\eps,\sg)^2\,|X_{i,\eps}(\sg)|\,d\eps,
\]
which tends to 
\[
2\int_{\Ga_j}u^2\,dl\quad\text{when } \sg\to 0
\] 
by applying the dominated convergence theorem together with Lemma~\ref{lem:coor} (vi) and the fact that $X_{i,\eps}(0)=\dot{\Ga}_j(\eps)$. This proves \eqref{eq:spirit23} and concludes the proof in \emph{Case 1}.

\vspace{0,1cm}
\emph{Case 2}. Suppose that $Z(u)=0$ in $E_j-\Ga_j$ for any $j=1,\ldots,m$. 

The proof of this case will follow from the previous one by means of an approximation argument.  For any $\sg\in (0,1)$ consider the set $D_\sg:=\{p\in\Sg\,;\,|\escpr{N_p,T_p}|>1-\sg\}$. Clearly $D_\sg$ is open in $\Sg$ and contains $\Sg_0$. We define $\phi_\sg:\Sg\to [0,1]$ by
\[
\phi_\sg:=
\begin{cases}
|\escpr{N,T}|,\quad\text{in } \overline{D}_\sg,
\\
1-\sg, \qquad\,\hspace{0,03cm}\text{in } \Sg-D_\sg.
\end{cases}
\]
The function $\phi_\sg$ is continuous, piecewise $C^1$ on $\Sg$ and piecewise $C^2$ on $\Sg-\Sg_0$. The sequence $\{\phi_\sg\}_{\sg\in (0,1)}$ equals $1$ on $\Sg_0$ and pointwise converges to $1$ when $\sg\to 0$. From the monotone convergence theorem, we deduce
\[
\lim_{\sg\to 0}\,\int_\Sg\mnh^{-1}\,Z(\phi_\sg)^2\,da=0,
\]
since the characteristic functions of $\overline{D}_\sg$ provide a non-decreasing sequence which pointwise converges to $0$ in $\Sg-\Sg_0$. Now, we modify $\phi_\sg$ around $\ptl D_\sg$ to get a sequence $\{\psi_\sg\}_{\sg\in(0,1)}$ of functions in $C^1(\Sg)\cap C^2(\Sg-\Sg_0)$ with the same properties. In particular, we have $\psi_\sg=|\escpr{N,T}|$ within an open neighborhood $D_\sg'$ of $\Sg_0$.

Let $w_\sg:=\psi_\sg\,u$. We obtain a sequence $\{w_\sg\}_{\sg\in(0,1)}$ in $C^1_0(\Sg)\cap C^2(\Sg-\Sg_0)$ which pointwise converges to $u$ in $\Sg$, has support contained in the support $C$ of $u$, and satisfies $w_\sg=u$ on $\Sg_0$. From the dominated convergence theorem, the Cauchy-Schwarz inequality and the fact that $\mnh^{-1}\,q$ extends continuously to $\Lambda$ by Lemma~\ref{lem:coor} (vi), we can show that $\{\mathcal{Q}(w_\sg)\}\to\mathcal{Q}(u)$ when $\sg\to 0$. For any $j=1,\ldots,m$ we choose a $\la_j$-neighborhood $(E_j)_\sg$ of $\Ga_j$ such that $(E_j)_\sg\subseteq E_j$ and $(E_j)_\sg\cap C\subseteq D'_\sg$. As a consequence $Z(w_\sg/\escpr{N,T})=Z(\pm u)=0$ in $(E_j)_\sg-\Ga_j$ for any $j=1,\ldots,m$.

Finally, consider $\vartheta\in C^1_0(\Sg)\cap C^2(\Sg-\Sg_0)$ such that $\int_\Sg \vartheta\,da=-1$, the restriction of $\vartheta$ to $\Sg_0$ is $C^2$ with compact support on $\Lambda$, and $Z(\vartheta/\escpr{N,T})=0$ on $E_j-\Ga_j$ for any $j=1,\ldots,m$. We define $u_\sg:=w_\sg+\alpha_\sg\,\vartheta$, where $\alpha_\sigma:=\int_\Sg w_\sigma\,da$. This gives a sequence $\{u_\sg\}_{\sg\in(0,\sg_0)}$ in $C^1_0(\Sg)\cap C^2(\Sg-\Sg_0)$ such that $\int_\Sg u_\sg\,da=0$ and $Z(u_\sg/\escpr{N,T})=0$ on $(E_j)_\sg-\Ga_j$, for any $j=1,\ldots,m$. Hence, we can apply the theorem for \emph{Case 1} to infer $\mathcal{Q}(u_\sg)\geq 0$ for any $\sg\in(0,1)$. By passing to the limit and using that $\{\alpha_\sg\}\to 0$ when $\sg\to 0$, we conclude that $\mathcal{Q}(u)\geq 0$.
\end{proof}

\section{Instability criterion and classification results}
\label{sec:main}
\setcounter{equation}{0} 

In this section we discuss the stability of the surfaces $\cmula$ introduced in Section~\ref{subsec:cmula} for any $3$-dimensional space form $M$. There are some previous related results, specially in the minimal case. In the Heisenberg group $\mm(0)$ a surface $\mathcal{C}_0(\Ga)$ is congruent to the hyperbolic paraboloid $t=xy$ or to a left-handed minimal helicoid, see \cite[Sect.~6]{rr2}. In the first case $\mathcal{C}_0(\Ga)$ is stable and, indeed, area-minimizing by a calibration argument~\cite[Thm.~5.3]{rr2}. In the second case $\mathcal{C}_0(\Ga)$ is unstable as an area-stationary surface~\cite[Thm.~5.4]{hrr}. On the other hand, Galli analyzed the stability of area-stationary $C^2$ surfaces with singular curves in the roto-translation group~\cite[Prop.~10.7, Prop.~10.9]{galli} and in the space of rigid motions of the Minkowski plane~\cite[Cor.~5.6, Cor.~5.8]{galli2}. More recently, the authors proved in \cite[Thm.~5.8]{hr3} that all the surfaces $\cmula$ in arbitrary $3$-dimensional space forms are stable under volume-preserving variations \emph{supported on the regular set}. 

In our main result below we produce deformations \emph{moving the singular curves} to show the instability of $\cmula$ under additional conditions. The precise statement is the following.

\begin{theorem}
\label{th:main}
Consider an embedded $C^2$ surface $\cmula$ in a $3$-dimensional space form $M$ of Webster curvature $\kappa$. If we suppose that $\la^2+\kappa\geq 1$ and that the length $\ell$ of $\Gamma$ satisfies $\ell>\sqrt{2}\,\pi$ when $\Ga$ is a circle, then $\cmula$ is unstable.
\end{theorem}

\begin{proof}
We will use the notation in Section~\ref{subsec:cmula}. We denote $\Sg:=\cmula$ and choose the unit normal $N$ on $\Sg$ for which $N=T$ along $\Ga$. For any $i=1,2$, let $\Ga_i$ be the singular curve of $\Sg$ associated to the cut constant $s_i\in (0,\pi/\sqrt{\la^2+\kappa})$. The restriction of $N$ to $\Sg_{i,\la}(\Ga)$ provides the unit normal such that $H=\la$ and the CC-geodesic rays $\ga_{i,\eps}(s)$ with $s\in (0,s_i)$ are characteristic curves. Note that $\Ga\neq\Ga_i$ because $N=-T$ along $\Ga_i$ for any $i=1,2$. 

From Lemma~\ref{lem:embedded} (i) all the singular curves are simultaneously injective curves or circles of the same length. In the first case we fix any number $\ell>0$, and consider any smooth function $\phi:\rr\to\rr$ with support contained in $[0,\ell]$. In the second case, $\ell$ stands for the length of $\Ga$ and we take any smooth $\ell$-periodic function $\phi:\rr\to\rr$. Anyway, we also impose the condition that $\int_0^\ell\phi(\eps)\,d\eps=0$. 

To show the instability of $\Sg$ we will employ a suitable test function $u$ in the stability inequality of Theorem~\ref{th:staineq}. For the construction of $u$ we distinguish two situations.

\vspace{0,1cm}
\emph{Case 1.} We assume that $\Ga_1=\Ga_2$.

We know that $\Sg=\Sg_1\cup\Sg_2$, where $\Sg_i:=\Sg_{i,\la}(\Ga)$ for any $i=1,2$. Since $\Ga_1=\Ga_2$ and both curves are parameterized by CC-geodesics of the same curvature, we can find $\eps_0\in\rr$
such that $\Ga_2(\eps)=\Ga_1(\eps+\eps_0)$ for any $\eps\in\rr$. Moreover $\eps_0\in [0,\ell]$ when $\Ga$ is a circle.

Take a value $\sg>0$ with $\sg<\min\{s_1/2,s_2/2\}$. With respect to the coordinates $(\eps,s)$, $s\in[0,s_i]$ appearing in Lemma~\ref{lem:coor} we define $w_\sg:\Sg_1\to\rr$ by
\[
w_\sg(\eps,s):=
\left\{ 
\begin{array}{ll}
\displaystyle{\phi(\eps)\,\escpr{N,T}(\sg)},
& 0 \leq s\leq\sg,
\\
\phi(\eps)\,\escpr{N,T}(s), & \sg\leq s\leq s_1/2,
\\
-\phi(\eps-\eps_0)\,\escpr{N,T}(s), & s_1/2\leq s\leq s_1-\sg,
\\
\phi(\eps-\eps_0)\,\escpr{N,T}(\sg), & s_1-\sg\leq s \leq s_1,
\end{array} 
\right.
\]
and $w_\sg:\Sg_2\to\rr$ by
\[
w_{\sg}(\eps,s):=\phi(\eps)\,\escpr{N,T}(\sg), \quad 0\leq s\leq s_2.
\]
Here $\escpr{N,T}(s)$ denotes the value of $\escpr{N,T}$ along the curve $\Ga_s\sub\Sg_1$ described by the coordinates $(\eps,s)$ when we fix $s\in[0,s_1]$. Observe that $\escpr{N,T}(s_1/2)=0$ by Lemma~\ref{lem:corbero} (iii). Thanks to statements (ii) and (iii) in Lemma~\ref{lem:embedded}, and the $\ell$-periodicity of $\phi$ when $\Ga$ is a circle, we infer that $w_\sg:\Sg\to\rr$ is a well-defined continuous function with compact support. It is clear that $w_\sg$ is piecewise $C^1$ in $\Sg$ and piecewise $C^\infty$ in $\Sg-\Sg_0$. Note also that the restriction of $w_\sg$ to the singular set $\Sg_0=\Ga\cup\Ga_1$ is a $C^\infty$ function with compact support on $\Sg_0$. Moreover, $w_\sg$ is $C^1$ around $\Sg_0$ and equality $Z(w_\sg)=0$ holds in the union of a $\la$-neighborhood of $\Ga$ with a $(-\la)$-neighborhood of $\Ga_1=\Ga_2$.  By using Fubini's theorem, the equality $\int_0^\ell\phi(\eps)\,d\eps=0$, and the fact deduced from Lemma~\ref{lem:coor} (iv) and equation \eqref{eq:vertimodel} that $da_i=j_i(s)\,d\eps\,ds$ for some function $j_i(s)$, we get $\int_\Sg w_\sg\,da=0$.

Next, we show that $\lim_{\sg\to 0}\mathcal{Q}(w_\sg)<0$ for the quadratic form $\mathcal{Q}$ defined in \eqref{eq:indexform} as
\[
\mathcal{Q}(w_\sg):=\int_\Sg\mnh^{-1}\big(Z(w_\sg)^2-q\,w_\sg^2\big)\,da+\int_{\Sg_0}\big(S(w_\sg)^2-4\,w_\sg^2\big)\,dl.
\]

From Lemma~\ref{lem:coor} (v) the extension to the singular curves of the vector field $S$ coincides, up to sign, with the tangent vector along these curves. By the definition of $w_\sg$ we obtain
\[
\int_{\Sg_0}\big(S(w_\sg)^2-4\,w_\sg^2\big)\,dl=2\,\escpr{N,T}^2(\sg)\,\int_0^\ell\big(\phi'(\eps)^2-4\,\phi(\eps)^2\big)\,d\eps,
\]
so that
\begin{equation}
\label{eq:serena1}
\lim_{\sg\to 0}\int_{\Sg_0}\big(S(w_\sg)^2-4\,w_\sg^2\big)\,dl=2\,\int_0^\ell\big(\phi'(\eps)^2-4\,\phi(\eps)^2\big)\,d\eps.
\end{equation}

Now, we compute the first integral in $\mathcal{Q}(w_\sg)$. Since $\Sg=\Sg_1\cup\Sg_2$ and $\Sg_1\cap\Sg_2=\Ga\cup\Ga_1$ we can divide the integral into two summands. The fact that $Z(w_\sg)=0$ on $\Sg_2$ implies that
\[
\int_{\Sg_2}\mnh^{-1}\big(Z(w_\sg)^2-q\,w_\sg^2\big)\,da=-C\,\escpr{N,T}^2(\sg)\,\int_0^\ell\phi(\eps)^2\,d\eps,
\]
where $C$ is the constant defined by
\begin{equation}
\label{eq:constantc}
C:=\int_0^{s_2}\mnh^{-1}(s)\,q(s)\,j_2(s)\,ds.
\end{equation}
Here $\mnh(s)$ and $q(s)$ denote the values of $\mnh$ and $q$ in coordinates $(\eps,s)$ with $s\in[0,s_2]$ (these only depend on $s$ by Lemma~\ref{lem:corbero} (i)). From Lemma~\ref{lem:coor} (vi) the function $\mnh^{-1}(s)\,q(s)$ extends continuously to $[0,s_2]$, and so $C$ is finite. By equation \eqref{eq:mnhq} we have
\[
\mnh^{-1}\,q=\mnh^{-1}\,\big(1+\escpr{B(Z),S}\big)^2+4\,(\la^2+\kappa-1)\,\mnh,
\]
so that $C\geq 0$ because $\la^2+\kappa\geq 1$. By passing to the limit when $\sg\to 0$, it follows that
\begin{equation}
\label{eq:serena21}
\lim_{\sg\to 0}\,\int_{\Sg_2}\mnh^{-1}\big(Z(w_\sg)^2-q\,w_\sg^2\big)\,da=-C\,\int_0^\ell\phi(\eps)^2\,d\eps.
\end{equation}
Next, we consider the $C^\infty$ surface $\Sg_\sg$ with empty singular set described by the coordinates $(\eps,s)$ with $s\in[\sg,s_1-\sg]$. Observe that $Z(w_\sg)=0$ on $\Sg_1-\Sg_\sg$. As a consequence
\begin{align*}
&\int_{\Sg_1-\Sg_\sg}\mnh^{-1}\big(Z(w_\sg)^2-q\,w_\sg^2\big)\,da=-\int_{\Sg_1-\Sg_\sg}\mnh^{-1}\,q\,w_\sg^2\,da
\\
&=-\escpr{N,T}^2(\sg)\,\bigg(\int_0^\ell\phi(\eps)^2\,d\eps\bigg)\,\bigg(\int_{[0,\sg]\cup[s_1-\sg,s_1]}\mnh^{-1}(s)\,q(s)\,j_1(s)\,ds\bigg).
\end{align*}
By using again that $\mnh^{-1}(s)\,q(s)$ extends continuously to the singular curves, we infer
\begin{equation}
\label{eq:serena22}
\lim_{\sg\to 0}\,\int_{\Sg_1-\Sg_\sg}\mnh^{-1}\big(Z(w_\sg)^2-q\,w_\sg^2\big)\,da=0.
\end{equation}
On the other hand, an application of Lemma~\ref{lem:aux} below to the surfaces $\Sg_{\sg,s_1/2}$ and $\Sg_{s_1/2,s_1-\sg}$ yields
\[
\int_{\Sg_\sg}\mnh^{-1}\big(Z(w_\sg)^2-q\,w_\sg^2\big)\,da=\frac{1}{\ell}\,\bigg(\int_0^\ell\phi(\eps)^2\,d\eps\bigg)\,\frac{\escpr{N,T}\,\big(\escpr{B(Z),S}-1\big)\,L_0^\ell(\Ga_s)}{\sqrt{1+\la^2\,\mnh^2}}\bigg|_\sg^{s_1-\sg}.
\] 
From here, and taking into account that $\escpr{B(Z),S}\to -1$ when we approach a singular curve by Lemma~\ref{lem:coor} (vi), we get
\begin{equation}
\label{eq:serena23}
\lim_{\sg\to 0}\,\int_{\Sg_\sg}\mnh^{-1}\big(Z(w_\sg)^2-q\,w_\sg^2\big)\,da=4\,\int_0^\ell\phi(\eps)^2\,d\eps.
\end{equation}

Having in mind \eqref{eq:serena1}, \eqref{eq:serena21}, \eqref{eq:serena22} and \eqref{eq:serena23}, we deduce that
\[
\lim_{\sg\to 0}\mathcal{Q}(w_\sg)=2\,\int_0^\ell\phi'(\eps)^2\,d\eps-(C+4)\,\int_0^\ell\phi(\eps)^2\,d\eps.
\]
Recall that $\ell$ is any positive number when $\Ga$ is injective. Since $C\geq 0$ and, by Wirtinger's inequality
\[
\inf\left\{\frac{\int_0^\ell\phi'(\eps)^2\,d\eps}{\int_0^\ell\phi(\eps)^2\,d\eps}\,;\,\phi\in C^\infty(\rr),\phi\neq0,\text{supp}(\phi)\subseteq[0,\ell],\int_0^\ell\phi(\eps)\,d\eps=0\right\}=\frac{4\pi^2}{\ell^2},
\]
we can choose $\ell$ and $\phi(\eps)$ in such a way that $\lim_{\sg\to 0}\mathcal{Q}(w_\sg)<0$. When $\Ga$ is a circle of length $\ell$, by taking $\phi(\eps):=\sin(2\pi\eps/\ell)$, it follows that
\[
\lim_{\sg\to 0}\mathcal{Q}(w_\sg)=\frac{4\pi^2}{\ell}-\frac{(C+4)\,\ell}{2},
\] 
which is negative by the hypothesis $\ell>\sqrt{2}\,\pi$ and the fact that $C\geq 0$. From an approximation argument similar to that in \emph{Case 2} of Theorem~\ref{th:staineq}, we can modify $w_\sg$ around the curves $\Ga_s$ with $s\in\{\sg,s_1/2,s_1-\sg\}$ to produce a function $u\in C^1_0(\Sg)\cap C^2(\Sg-\Sg_0)$ with $\int_\Sg u\,da=0$ and $\mathcal{Q}(u)<0$. Thus, we can invoke Theorem~\ref{th:staineq} to conclude that $\Sg$ is unstable, as desired.

\vspace{0,1cm}
\emph{Case 2.} We suppose that $\Ga_1\neq\Ga_2$.

In this case $\Sg_0$ contains at least three different singular curves $\Ga$, $\Ga_1$ and $\Ga_2$. We consider the pieces of $\Sg$ given by $\Sg_i:=\Sg_{i,\la}(\Ga)$ for any $i=1,2$ and $\Sg_3=\Sg_{1,-\la}(\Ga_2)$. Observe that the functions in \eqref{eq:vertimodel} coincide for the surfaces $\Sg_1$ and $\Sg_3$. In particular, the associated cut constants also coincide and so, we have coordinates $(\eps,s)$ with $s\in[0,s_1]$ to describe both $\Sg_1$ and $\Sg_3$. To avoid confusions, in the construction below we will use $N_i$ with $i=1,3$ to denote the restriction of $N$ to $\Sg_i$.

We will find a test function $w_\sg$ with compact support in $\cup_{i=1}^3\Sg_i$ and such that $w_\sg\neq 0$ along $\Ga\cup\Ga_2$. For any $\sg\in (0,s_1/2)$ we define $w_\sg:\Sg_1\to\rr$ in the coordinates $(\eps,s)$ with $s\in[0,s_1]$ by
\[
w_\sg(\eps,s):= 
\left\{ \begin{array}{ll}
\displaystyle{\phi(\eps)\,\escpr{N_1,T}(\sg)},
& 0 \leq s\leq \sg,
\\
\phi(\eps)\,\escpr{N_1,T}(s), & \sg \leq s \leq s_1/2 ,
\\
0, & s_1/2 \leq s \leq s_1,
\end{array} \right.
\]
where $\escpr{N_1,T}(s)$ is the value of $\escpr{N,T}$ along the curve of $\Sg_1$ associated to the coordinates $(\eps,s)$ when we fix $s\in[0,s_1]$. We also define
$w_\sg:\Sg_2\to\rr$ by
\[
w_{\sg}(\eps,s):=\phi(\eps)\,\escpr{N_1,T}(\sg), \quad 0\leq s\leq s_2,
\]
and $w_\sg:\Sg_3\to\rr$ by
\[
w_\sg(\eps,s):= 
\left\{ \begin{array}{ll}
\displaystyle{\phi(\eps)\,\escpr{N_1,T}(\sg)},
& 0 \leq s\leq \sg,
\\
-\phi(\eps)\,\escpr{N_3,T}(s), & \sg\leq s \leq s_1/2,
\\
0, & s_1/2 \leq s \leq s_1,
\end{array} \right.
\]
where $\escpr{N_3,T}(s)$ is the value of $\escpr{N,T}$ along the curve in $\Sg_3$ having coordinates $(\eps,s)$ with $s\in[0,s_1]$ fixed. We extend $w_\sg$ to the whole surface $\Sg$ by setting $w_\sg=0$ in $\Sg-\cup_{i=1}^3\Sg_i$. 

Note that $\escpr{N_1,T}(s_1/2)=0$ by Lemma~\ref{lem:corbero} (iii), and that $\escpr{N_1,T}(s)=-\escpr{N_3,T}(s)$ for any $s\in[0,s_1]$ by the expressions of $N_1$ and $N_3$ appearing in \eqref{eq:normalcoor} and \eqref{eq:normal3}. By taking into account statements (ii) and (iv) in Lemma~\ref{lem:embedded}, the function $w_\sg:\Sg\to\rr$ is well defined and continuous. Moreover, it is piecewise $C^1$ in $\Sg$ and piecewise $C^\infty$ in $\Sg-\Sg_0$. Around the singular set $w_\sg$ is $C^1$ and satisfies $Z(w_\sg)=0$. The fact that $\int_0^\ell\phi(\eps)\,d\eps=0$ implies that $\int_\Sg w_\sg\,da=0$.

Now, we can proceed as in \emph{Case 1} to compute $\lim_{\sg\to 0}\mathcal{Q}(w_\sg)$. The formulas \eqref{eq:serena1} and \eqref{eq:serena21} still holds in this case. On the other hand, with the help of Lemma~\ref{lem:aux} below, we obtain
\[
\lim_{\sg\to 0}\,\int_{\Sg_1}\mnh^{-1}\big(Z(w_\sg)^2-q\,w_\sg^2\big)\,da=\lim_{\sg\to 0}\,\int_{\Sg_3}\mnh^{-1}\big(Z(w_\sg)^2-q\,w_\sg^2\big)\,da=2\,\int_0^\ell\phi(\eps)^2\,d\eps.
\]
By combining everything, we arrive at
\[
\lim_{\sg\to 0}\mathcal{Q}(w_\sg)=2\,\int_0^\ell\phi'(\eps)^2\,d\eps-(C+4)\,\int_0^\ell\phi(\eps)^2\,d\eps.
\]
From this point we can reason as in \emph{Case 1} to deduce the instability of $\Sg$.
\end{proof}

\begin{lemma}
\label{lem:aux}
For given values $0<a<b<s_1$, let $\Sg_{a,b}$ be the portion of $\Sg_1$ associated to the coordinates $(\eps,s)$ with $s\in[a,b]$. Let $w:\Sg_{a,b}\to\rr$ defined by $w(\eps,s):=\phi(\eps)\,\escpr{N,T}(s)$, for some function $\phi\in C^1(\rr)$. Suppose that $\phi$ has support contained in an interval $[\alpha,\beta]$ of length $\ell$ when $\Ga$ is injective, or that it is $\ell$-periodic when $\Ga$ is a circle of length $\ell$. Then, we have
\[
\int_{\Sg_{a,b}}\mnh^{-1}\big(Z(w)^2-q\,w^2\big)\,da=\frac{1}{\ell}\,\bigg(\int_\alpha^\beta\phi(\eps)^2\,d\eps\bigg)\,\frac{\escpr{N,T}\,\big(\escpr{B(Z),S}-1\big)\,L_0^\ell(\Ga_s)}{\sqrt{1+\la^2\,\mnh^2}}\bigg|_a^b,
\]
where $L_0^\ell(\Ga_s)$ denotes the length in $[0,\ell]$ of the curve $\Ga_s$ described by the coordinates $(\eps,s)$ when we fix $s\in[a,b]$.
\end{lemma}

\begin{proof}
First, we need to show that identity
\begin{equation}
\label{eq:dfred}
\divv_\Sg \big(\escpr{N,T}\,\big(\escpr{B(Z),S}-1\big)\,Z\big)=
\mnh^{-1}\,\big(Z(\escpr{N,T})^2-q\,\escpr{N,T}^2\big)
\end{equation}
holds for any oriented CMC surface $\Sg$ of class $C^2$ having regular set $\Sg-\Sg_0$ of class $C^3$. For this, we consider the second order operator
\[
\mathcal{L}(\psi):=\mnh^{-1}\,\big(Z(Z(\psi))+2\,\mnh^{-1}\,\escpr{N,T}\,\escpr{B(Z),S}\,Z(\psi)+q\,\psi\big).
\]
From the expressions of $Z(\mnh)$ and $\divv_\Sg Z$ in \eqref{eq:zder} and \eqref{eq:divzs}, we get
\[
\mathcal{L}(\psi)=\divv_\Sg\big(\mnh^{-1}\,Z(\psi)\,Z\big)+\mnh^{-1}\,q\,\psi.
\] 
On the other hand, it was proved in \cite[Lem.~3.4]{hr3} that $\mathcal{L}(\escpr{N,T})=0$, so that
\[
\divv_\Sg\big(\mnh^{-1}\,Z(\escpr{N,T})\,Z\big)=-\mnh^{-1}\,q\,\escpr{N,T}.
\]
Therefore, for any function $u\in C^1(\Sg)$ we deduce
\begin{align*}
\divv_\Sg \big(\mnh^{-1}\,u\,Z(\escpr{N,T})\,Z\big)=-\mnh^{-1}\,q\,\escpr{N,T}\,u+\mnh^{-1}\,Z(\escpr{N,T})\,Z(u),
\end{align*}
which provides \eqref{eq:dfred} when we choose $u=\escpr{N,T}$ and use \eqref{eq:zder}.

Now, we prove the formula in the statement. Let $\Sg^*_{a,b}$ be the subset of $\Sg_{a,b}$ where $\eps\in[0,\ell]$. By taking into account \eqref{eq:dfred} and applying the Riemannian divergence theorem, we obtain
\begin{align*}
&\int_{\Sg_{a,b}}\mnh^{-1}\big(Z(w)^2-q\,w^2\big)\,da
\\
&=\bigg(\int_\alpha^\beta\phi(\eps)^2\,d\eps\bigg)\bigg(\int_a^b\mnh^{-1}(s)\,\big(\escpr{N,T}'(s)^2-q(s)\,\escpr{N,T}^2(s)\big)\,j_1(s)\,ds\bigg)
\\
&=\frac{1}{\ell}\,\bigg(\int_\alpha^\beta\phi(\eps)^2\,d\eps\bigg)\bigg(\int_{\Sg^*_{a,b}}\mnh^{-1}\,\big(Z(\escpr{N,T})^2-q\,\escpr{N,T}^2\big)\,da\bigg)
\\
&=\frac{1}{\ell}\,\bigg(\int_\alpha^\beta\phi(\eps)^2\,d\eps\bigg)\bigg(\int_{\Sg^*_{a,b}}\divv_\Sg \big(\escpr{N,T}\,\big(\escpr{B(Z),S}-1\big)\,Z\big)\,da\bigg)
\\
&=\frac{-1}{\ell}\,\bigg(\int_\alpha^\beta\phi(\eps)^2\,d\eps\bigg)\,\int_{\ptl\Sg^*_{a,b}}\escpr{N,T}\,\big(\escpr{B(Z),S}-1\big)\,\escpr{Z,\nu}\,dl,
\end{align*}
where $\nu$ stands for the unit conormal along $\ptl\Sg^*_{a,b}$ pointing into $\Sg_{a,b}$. Observe that $\escpr{Z,\nu}=0$ along the portion of $\ptl\Sg^*_{a,b}$ contained inside characteristic segments (which is empty when $\Ga$ is a circle of length $\ell$). Therefore, the desired formula comes from the expression of $\nu$ in Lemma~\ref{lem:conormal}. 
\end{proof}

\begin{remark}
In the proof of the theorem the hypotheses $\la^2+\kappa\geq 1$ and $\ell>\sqrt{2}\,\pi$ guarantee that some of the considered functions $w_\sg$ satisfy $\mathcal{Q}(w_\sg)<0$. The first hypothesis is only used to prove that the constant $C$ in \eqref{eq:constantc} is nonnegative. This condition may fail if $\la^2+\kappa<1$. For instance, the surface $\cmula$ in $\mm(0)$ obtained when $\Ga$ is a horizontal line and $\la>0$ satisfies that $C_\la\to-\infty$ when $\la\to 0$. It is also natural to ask if the second hypothesis is necessary. In Example~\ref{ex:cornucopia} we show some evidence of the existence of a stable surface $\cmula$ where $\Ga$ is a circle of length $\ell\leq\sqrt{2}\,\pi$.
\end{remark}

The instability criterion in Theorem~\ref{th:main} can be combined with previous characterization and stability results to deduce the classification of stable and embedded $C^2$ surfaces in the simply connected $3$-dimensional space forms of non-negative Webster curvature.

\begin{corollary}
\label{cor:stablesphere}
Let $\Sg$ be a complete, connected, oriented and embedded $C^2$ surface in the sub-Riemannian sphere $\mm(1)$. If $\Sg$ is stable, then $\Sg$ is a Pansu spherical surface.
\end{corollary}

\begin{proof} 
Note that the singular set $\Sg_0$ of $\Sg$ cannot be empty; otherwise, we would deduce from \cite[Cor.~6.9]{rosales} that $\Sg$ is unstable. Thus, it follows from Theorem~\ref{th:structure} (ii) that $\Sg_0$ must contain an isolated point or a $C^1$ curve. In the first case the authors proved in \cite[Thm.~5.3]{hr1}, see also \cite[Thm.~4.9]{hr2}, that $\Sg$ must be a Pansu spherical surface. In the second case, we know from Theorem~\ref{th:cmula} that $\Sg=\cmula$ for some CC-geodesic $\Ga$ in $\mm(1)$. As we pointed out in Remark~\ref{re:cosillas} the curve $\Ga$ must be a circle because $\Sg$ is $C^2$ around $\Sg_0$. Hence, the length $\ell$ of $\Ga$ satisfies $\ell\geq 2\pi$ by \cite[Prop.~2.5]{hr2}. Now, we can apply Theorem~\ref{th:main} to conclude that $\Sg$ is unstable.
\end{proof}

\begin{corollary}
\label{cor:stableh1}
Let $\Sg$ be a complete, connected, oriented and embedded $C^2$ surface in the Heisenberg group $\mm(0)$. If $\Sg$ is stable, then $\Sg$ is a Euclidean plane, a Pansu sphere or a surface $\mathcal{C}_0(\Ga)$ with $\Ga$ a horizontal line.
\end{corollary}

\begin{proof}
If $\Sg_0=\emptyset$ then the stability condition implies that $\Sg$ is a vertical plane by \cite[Cor.~6.9]{rosales}. In case $\Sg_0\neq\emptyset$ the characterization results for volume-preserving area-stationary $C^2$ surfaces in $\mm(0)$, see \cite[Sect.~6]{rr2} and also \cite[Sec.~4.2]{hr2}, imply that $\Sg$ is either a Euclidean horizontal plane, a Pansu sphere, or a surface $\cmula$. Let us analyze the case $\Sg=\cmula$.

When $\la=0$ it is known \cite[Thm.~6.15]{rr2} that, either $\Ga$ is a horizontal line and $\Sg$ is congruent to the hyperbolic paraboloid $t=xy$, or $\Ga$ is a helix and $\Sg$ is congruent to a left-handed minimal helicoid. It was proved in \cite[Thm.~5.4]{hrr} that such helicoids are unstable as area-stationary surfaces, i.e., under compactly supported variations that need not preserve volume. By using global coordinates $(\eps,s)\in\rr^2$ of a helicoid, and a suitable function $\psi\in C^\infty_0(\rr)$, it was possible to find a test function $u(\eps,s):=\phi(\eps)\,\psi(s)$ satisfying $\mathcal{Q}(u)<0$, for any $\phi\in C_0^\infty(\rr)$ with support $[-\eps_0,\eps_0]$ and $\eps_0>0$ big enough. If we also require $\phi$ to have mean zero, then the same proof shows that the left-handed helicoids are also unstable as volume-preserving area-stationary surfaces. 

Finally, consider the case $\la\neq 0$. Let $\delta_r$ be the anisotropic dilation in $\mm(0)$ given by 
\[
\delta_r(x,y,t):=(e^r\,x,e^r\,y,e^{2r}\,t).
\]
For any $C^2$ surface $\Sg$, it is well known that $(\delta_r(\Sg))_0=\delta_r(\Sg_0)$, $A(\delta_r(\Sg))=e^{3r}\,A(\Sg)$, and the mean curvature of $\delta_r(\Sg)$ equals $e^{-r}\,H$, see for instance \cite{rr2}. On the other hand, as $\delta_r^*(dv)=e^{4r}\,dv$, we can proceed as in \cite[Lem.~3.2]{hrr} to deduce that $\Sg$ is stable if and only if $\delta_r(\Sg)$ is stable. By choosing $r:=\log(|\la|)$ we get that the stability of $\Sg=\cmula$ is equivalent to the stability of $\mathcal{C}_{\pm 1}(\Psi)$, where $\Psi:=\delta_r(\Ga)$. Since none of the CC-geodesics in $\mm(0)$ is  a circle we can invoke Theorem~\ref{th:main} to conclude the instability of $\mathcal{C}_{\pm 1}(\Psi)$.
\end{proof}

\begin{remark}
In the previous corollaries the converse statements are also true in the following sense. In $\mm(0)$ a calibration argument, see \cite[Thm.~5.3]{rr2} and \cite[Thm.~2.3]{bscv}, shows that the Euclidean planes and the surfaces $\mathcal{C}_0(\Ga)$ with $\Ga$ a horizontal line are area-minimizing and, in particular, stable. On the other hand, the authors proved in \cite[Thm.~5.9]{hr2} that the Pansu spheres of any $3$-space form are stable under volume-preserving admissible variations which are $C^3$ off the poles.
\end{remark}

\section{The isoperimetric problem in the sub-Riemannian $3$-sphere}
\label{sec:isoperimetric}
\setcounter{equation}{0} 

We finish this work with a uniqueness theorem for $C^2$ isoperimetric regions in $\mm(1)$. First, we recall some elementary definitions and facts about the \emph{isoperimetric problem}.

Let $M$ be a Sasakian sub-Riemannian $3$-manifold. For any Borel set $\Om\subeq M$, the \emph{volume} of $\Om$ is the Riemannian volume $|\Om|$ in $(M,g)$.  Following \cite{fssc} we define the \emph{perimeter} of $\Om$ as
\begin{equation*}
P(\Om):=\sup\left\{\int_\Om\divv U\,dv;\,|U|\leq 1\right\},
\end{equation*}
where $\divv$ denotes the divergence operator in $(M,g)$ and $U$ ranges over horizontal $C^1$ vector fields with compact support on $M$. Observe that  $P(\Om)=A(\Sg)$ by the Riemannian divergence theorem when $\ptl\Om$ is a $C^2$ surface $\Sg$. An \emph{isoperimetric region} in $M$ is a set $\Om\subeq M$ such that $P(\Om)\leq P(\Om')$ for any other set $\Om'\subeq M$ with $|\Om'|=|\Om|$.

In the model spaces $\mm(\kappa)$ the existence of isoperimetric regions of any volume is guaranteed by the results of Leonardi and Rigot~\cite[Thm.~3.2]{lr}, and of Galli and Ritor\'e~\cite[Thm.~6.1]{galli-ritore}. Indeed, in the sub-Riemannian $3$-sphere $\mm(1)$, as in any compact contact sub-Riemannian manifold, the existence comes from the lower semicontinuity of the perimeter and a compactness result, see for instance \cite{miranda-bv} and \cite[Ch.~5]{survey}. The regularity of the solutions is still an unsolved question. 

If $\Om$ is a bounded $C^2$ isoperimetric region in $M$, then $\Sg:=\ptl\Om$ is a compact volume-preserving area-stationary surface. This follows since, for any variation $\var:I\times\Sg\to M$, the associated surfaces $\Sg_r:=\var_r(\Sg)$ satisfy $A(\Sg_r)=P(\Om_r)$ and the volume functional $|\Om_r|$ of the enclosed sets $\Om_r$ coincides, up to a constant, with the signed volume  defined in \eqref{eq:volume}. In $\mm(\kappa)$ this fact, together with an Alexandrov type theorem, see \cite[Thm.~6.10]{rr2} and \cite[Thm.~4.11]{hr2}, allows to conclude that $\Sg$ is a Pansu spherical surface when $\kappa\leq 0$. 

In $\mm(1)$ the previous scheme fails because there are many compact volume-preserving area-stationary surfaces besides the Pansu spheres. For instance, the Clifford tori $\sph^1(\rho)\times\sph^1(\sqrt{1-\rho^2})$ with $\rho\in(0,1)$ and the tori $\cmula$ where $\Ga$ is a horizontal great circle provide first order candidates to solve the isoperimetric problem. In order to discard these and other candidates it is natural to consider the stability condition. This leads us to the following result.

\begin{corollary}
\label{cor:isop}
If $\Om$ is a $C^2$ isoperimetric region in the sub-Riemannian $3$-sphere $\mm(1)$, then $\ptl\Om$ is a Pansu spherical surface.
\end{corollary}

\begin{proof}
We denote $\Sg:=\ptl\Om$, which is a compact and embedded $C^2$ surface in $\mm(1)$. As $\Om$ is an isoperimetric region then $\Sg$ is stable. In particular, $\Sg$ has constant mean curvature $H=\la$ with respect to the inner unit normal $N$. Moreover, we can invoke Corollary~\ref{cor:stablesphere} to deduce that any connected component of $\Sg$ is a Pansu sphere. To prove that $\Sg$ is connected we will employ a standard argument with some modifications due to the presence of isolated singular points.

Suppose that there were two connected components $\Sg_1$ and $\Sg_2$. We consider a function $u:\Sg\to\rr$ which is a constant $c_i\neq 0$ on $\Sg_i$, vanishes on $\Sg-(\Sg_1\cup\Sg_2)$, and satisfies $\int_\Sg u\,da=0$. From \cite[Lem.~8.3]{hr2} there is an open interval $I\sub\rr$ with $0\in I$ such that the variation $\psi:I\times\Sg\to\mm(1)$ defined by $\psi(r,p):=\exp_p(r\,u(p)\,N_p)$ is admissible. Choose any function $w\in C_0^\infty(\Sg_1)$ with $\int_\Sg w\,da\neq 0$ and supported on the regular set. By using Lemma~\ref{lem:bdc} we can modify $\psi$ to produce a volume-preserving variation $\var$ of $\Sg$ with velocity vector field $U:=uN$. This variation is still admissible since it coincides with $\psi$ near the singular points. Hence, we can apply the second variation formula for the Pansu spherical surfaces~\cite[Thm.~5.2]{hr2} to obtain
\[
A''(0)=\int_\Sg\mnh^{-1}\big\{Z(u)^2-\big(1+\la^2\mnh^2\big)^2u^2\big\}\,da=-\sum_{i=1}^2\int_{\Sg_i}\mnh^{-1}\big(1+\la^2\mnh^2\big)^2u^2\,da,
\]
which contradicts the stability of $\Sg$.
\end{proof}

Finally, we will prove a uniqueness result for the $C^2$ solutions of the isoperimetric problem in the sub-Riemannian model of the $3$-dimensional projective space. For that, we need to introduce some facts about this space.

Let $G$ be the subgroup of isometries of $\mm(1)$ given by $\{\text{Id},-\text{Id}\}$. The Sasakian structure and the quaternion multiplication in $\mm(1)$ descends to the quotient $\mathbb{RP}^3:=\mm(1)/G$. The associated projection $\Pi:\mm(1)\to\mathbb{RP}^3$ is a local isometry between sub-Riemannian $3$-manifolds and a covering map. Hence $\mathbb{RP}^3$ is a space form of constant Webster curvature $\kappa=1$.

Suppose that $\widetilde{\ga}:\rr\to\mathbb{RP}^3$ is a complete CC-geodesic of curvature $\la$. We can write $\widetilde{\ga}=\Pi\circ\ga$ for some complete CC-geodesic $\gamma$ in $\mm(1)$. From the expression of $\ga$ in \cite[Eq.~(3.5)]{hr1} we can show that $\widetilde{\ga}$ is injective when $\la/\sqrt{1+\la^2}\in\rr-\mathbb{Q}$ or a circle when $\la/\sqrt{1+\la^2}\in\mathbb{Q}$. This behaviour only depends on $\la$ and not on the initial conditions of $\widetilde{\ga}$.

\begin{corollary}
\label{cor:projective}
If $\Om$ is a $C^2$ isoperimetric region in $\mathbb{RP}^3$, then $\ptl\Om$ is a Pansu spherical surface or a compact embedded surface $\cmula$ for some CC-geodesic circle $\Ga$ of length $\ell\leq\sqrt{2}\,\pi$.
\end{corollary}

\begin{proof}
Take any connected component $\Sg'$ of the surface $\Sg:=\ptl\Om$. The aforementioned property of the CC-geodesics in $\mathbb{RP}^3$ allows us to apply the stability result in \cite[Thm.~6.7, Re.~6.8]{rosales} to infer that any complete, oriented and CMC surface of class $C^2$ in $\mathbb{RP}^3$ with empty singular set is unstable. Therefore $\Sg'_0\neq\emptyset$ because $\Sg$ is stable. By the characterization of volume-preserving area-stationary $C^2$ surfaces in Theorem~\ref{th:cmula} and \cite[Thm.~4.9]{hr2}, we get that $\Sg'$ is either a Pansu sphere or an embedded $C^2$ surface $\cmula$ for some complete CC-geodesic $\Ga$. If $\Ga$ is not a circle then it is an injective curve by Lemma~\ref{lem:corbero} (i). Thus $\Ga=\Pi\circ\ga$ for some complete and injective CC-geodesic $\ga$ in $\mm(1)$. By \cite[Prop.~3.3]{hr1} the trace of $\ga$ is a dense subset of a surface $\mathcal{T}\subset\mm(1)$ congruent to a Clifford torus. Hence, from the local diffeomorphism $\Pi:\mathcal{T}\to\cmula$, we would conclude that $\Pi(\mathcal{T})=\cmula$ which contradicts that $\mathcal{T}$ has empty singular set. Observe that $\ell\leq \sqrt{2}\,\pi$ by Theorem~\ref{th:main} since $\cmula$ is stable. To finish the proof it suffices to see that $\Sg$ is connected. This is achieved with the same arguments as in Corollary~\ref{cor:isop} combined with the stability inequality in Theorem~\ref{th:staineq}.
\end{proof}

As happens in other non-simply connected $3$-dimensional space forms~\cite[Ex.~6.2]{hr2} the Pansu spheres in $\mathbb{RP}^3$ do not always minimize the area for the enclosed volume. Indeed, motivated by the Riemannian situation described by Ritor\'e and Ros~\cite[Thm.~8]{rit-ros}, we may expect that the vertical Clifford tori of $\mathbb{RP}^3$ are isoperimetrically better than the Pansu spheres for a certain range of volumes. Let us see this in more detail.

\begin{example}
\label{ex:cornucopia}
Let $\sla$ be the Pansu sphere in $\mm(1)$ of constant mean curvature $\la\geq 0$ and south pole at the identity element for the quaternion product. Note that $\mathcal{S}_0$ is a totally geodesic $2$-sphere by \cite[Eq.~(3.5)]{hr1}. The area of $\sla$ can be computed by using the polar coordinates in \cite[Lem.~3.6]{hr2}, so that we get $A(\sla)=\pi^2/(1+\la^2)^{3/2}$. On the other hand, consider the Clifford torus $\mathcal{T}_\rho:=\sph^1(\rho)\times\sph^1(\sqrt{1-\rho^2})$ with $\rho\in(0,1)$. This is a vertical surface, which means that the Reeb vector field is always tangent to $\mathcal{T}_\rho$, see \cite[Ex.~4.6]{hr1}. Hence $\mnh=1$ along $\mathcal{T}_\rho$ and so, the Riemannian area of $\mathcal{T}_\rho$ coincides with the sub-Riemannian one. As $\mathcal{T}_\rho$ is antipodally symmetric then it descends naturally to $\mathbb{RP}^3$. The resulting torus has area $A(\Pi(\mathcal{T}_\rho))=2\pi^2\rho\,\sqrt{1-\rho^2}$ and divides $\mathbb{RP}^3$ into two domains of volumes $\pi^2\rho^2$ and $\pi^2\,(1-\rho^2)$. Since $\lim_{\rho\to 1}A(\Pi(T_\rho))=0$ and $\lim_{\la\to 0}A(\Pi(\mathcal{S}_\la))\neq 0$ we deduce that, for $\la$ close enough to $0$, the corresponding Pansu sphere in $\mathbb{RP}^3$ is isoperimetrically worse than the vertical Clifford torus of the same volume.
\end{example}

The previous comparison together with Corollary~\ref{cor:projective} implies that, by assuming $C^2$ regularity of the isoperimetric regions in $\mathbb{RP}^3$, there exists a solution bounded by a surface $\cmula$. This fact does not contradict Theorem~\ref{th:main} because there are CC-geodesic circles in $\mathbb{RP}^3$ of length $\ell\leq\sqrt{2}\,\pi$.

\providecommand{\bysame}{\leavevmode\hbox to3em{\hrulefill}\thinspace}
\providecommand{\MR}{\relax\ifhmode\unskip\space\fi MR }
\providecommand{\MRhref}[2]{%
  \href{http://www.ams.org/mathscinet-getitem?mr=#1}{#2}
}
\providecommand{\href}[2]{#2}

\end{document}